\documentclass[a4paper,oneside,11pt]{amsart}

\reversemarginpar

\usepackage[colorlinks,hyperindex,linkcolor=blue,urlcolor=black,pdftitle={On expansive operators that are quasisimilar
 to the unilateral shift of finite multiplicity},pdfauthor={Maria F.Gamal'}]
{hyperref}

\usepackage{amsmath}
\usepackage{amsfonts}
\usepackage{amssymb}
\usepackage{amsthm}

\usepackage[english]{babel}
\usepackage{enumerate}

\usepackage{graphicx}
\usepackage{color}

\usepackage[abbrev]{amsrefs}

\numberwithin{equation}{section}

\theoremstyle{plain}
\newtheorem{theorem}{Theorem}[section]
\newtheorem{lemma}[theorem]{Lemma}

\newtheorem{corollary}[theorem]{Corollary}
\newtheorem{proposition}[theorem]{Proposition}
\newtheorem{theoremcite}{Theorem}

\theoremstyle{definition}
\newtheorem{remark}[theorem]{Remark}
\newtheorem{example}[theorem]{Example}

\begin{document}

\title[On expansive operators]{On expansive operators that are quasisimilar
 to the unilateral shift of finite multiplicity}

\author{Maria F. Gamal'}
\address{
 St. Petersburg Branch\\ V. A. Steklov Institute 
of Mathematics\\
Russian Academy of Sciences\\ Fontanka 27, St. Petersburg\\ 
191023, Russia  
}
\email{gamal@pdmi.ras.ru}



\begin{abstract}
An operator $T$ on a Hilbert space $\mathcal H$ is called expansive, if $\|Tx\|\geq \|x\|$ ($x\in\mathcal H$). 
Expansive operators  $T$  quasisimilar to the unilateral shift $S_N$ of finite multiplicity $N$ are studied. 
It is proved that $I-T^*T$ is of trace class for such $T$. Also the lattice $\mathrm{Lat}T$ of invariant subspaces of an expansive operator  $T$ quasisimilar to $S_N$ is studied. It is proved that $\dim\mathcal M\ominus T\mathcal M\leq  N$ 
for  every $\mathcal M\in\mathrm{Lat}T$.  It is shown that if $N\geq 2$, then there exist $\mathcal M_j\in\mathrm{Lat}T$ ($j=1,\ldots, N$)  such that 
the restriction $T|_{\mathcal M_j}$ of $T$ on $\mathcal M_j$ is similar to the  unilateral shift $S$ of multiplicity $1$ 
for every $j=1,\ldots, N$, 
and $\mathcal H=\vee_{j=1}^N\mathcal M_j$. For $N=1$, that is, for $T$ quasisimilar to $S$, there exist two spaces 
 $\mathcal M_1$,  $\mathcal M_2\in\mathrm{Lat}T$  such that $T|_{\mathcal M_j}$ is similar to $S$  
for $j=1,2$, 
and $\mathcal H=\mathcal M_1\vee\mathcal M_2$. Example of an expansive operator $T$ quasisimilar to $S$ is given such that intertwining transformations do  not give an isomorphism of $\mathrm{Lat}T$ and $\mathrm{Lat}S$.


 \end{abstract}

\maketitle

\section{Introduction}

 Let $\mathcal H$ be a (complex, separable) Hilbert space, 
and let  $\mathcal L(\mathcal H)$ be the algebra of all (bounded, linear)  operators acting on  $\mathcal H$. A (closed) subspace  $\mathcal M$ of  $\mathcal H$ is called \emph{invariant} 
for an operator $T\in\mathcal L(\mathcal H)$, if $T\mathcal M\subset\mathcal M$.
 The complete lattice of all invariant subspaces of $T$ is denoted by  $\operatorname{Lat}T$. 

The \emph{multiplicity} $\mu_T$ of an operator  $T\in \mathcal L(\mathcal H)$ 
 is the minimum dimension of its reproducing subspaces: 
\begin{equation}\label{mu}  \mu_T=\min\{\dim E: E\subset \mathcal H, \ \ 
\vee_{n=0}^\infty T^n E=\mathcal H \}.\end{equation}

For Hilbert spaces  $\mathcal H$ and $\mathcal K$, let    $\mathcal L(\mathcal H, \mathcal K)$ denote the space of (bounded, linear) 
transformations acting from $\mathcal H$ to $\mathcal K$.
For $A\in\mathcal L(\mathcal H)$ and  $B\in\mathcal L(\mathcal K)$ 
set
\begin{equation*}
\mathcal I(A,B)=\{X\in\mathcal L(\mathcal H, \mathcal K)\ :\ XA=BX\}.
\end{equation*} 
Then $\mathcal I(A,B)$ is the set of all transformations $X$ which  \emph{intertwine} $A$ and $B$.
Let $X\in\mathcal I(A,B)$.  
If $X$ is unitary, then $A$ and $B$ 
are called  \emph{unitarily equivalent}, in notation: $A\cong B$. If $X$ is invertible, that is, $X^{-1}\in\mathcal L(\mathcal K, \mathcal H)$, 
then $A$ and $B$ are called \emph{similar}, in notation: $A\approx B$.
If $X$ is a \emph{quasiaffinity}, that is, $\ker X=\{0\}$ and $\operatorname{clos}X\mathcal H=\mathcal K$, then
$A$ is called a  \emph{quasiaffine transform} of $B$, in notation: $A\prec B$. 
If  $\ker X=\{0\}$, we write $A\buildrel i \over\prec B$. 
If  $\operatorname{clos}X\mathcal H=\mathcal K$, we write $A\buildrel d \over\prec B$. 
If $A\prec B$ and 
$B\prec A$, then $A$ and $B$ are called  \emph{quasisimilar}, in notation: $A\sim B$. 

An operator $T\in\mathcal L(\mathcal H)$ is called \emph{expansive} if $\|Tx\|\geq\|x\|$ for every $x\in \mathcal H$. 

An  operator $T$ is called  \emph{power bounded}, if $\sup_{n\geq 0}\|T^n\| < \infty$.
 An  operator $T$  is called a  \emph{contraction}, if $\|T\|\leq 1$. 
Clearly, a contraction is power bounded.

Let   $T\in\mathcal L(\mathcal H)$ be a power bounded operator. 
 It is easy to see that the space 
\begin{equation}\label{01}
\mathcal H_{T,0}=\{x\in\mathcal H \ :\ \|T^nx\|\to 0\}
\end{equation}
 is invariant for $T$ (cf. {\cite[Theorem II.5.4]{nfbk}}). Classes $C_{ab}$, $a$, $b=0,1,\cdot$, of power bounded operators 
are defined as follows (see  {\cite[Sec. II.4]{nfbk}} and \cite{ker89}). If $\mathcal H_{T,0}=\mathcal H$, then  $T$ is  \emph{of class} $C_{0\cdot}$, while if  $\mathcal H_{T,0}=\{0\}$, then $T$ is 
 \emph{of class} $C_{1\cdot}$. Furthermore,  $T$  is \emph{of class} $C_{\cdot a}$, if $T^\ast$ is of class  $C_{a\cdot}$,  
 and $T$ is  \emph{of class} $C_{ab}$, if $T$ is of classes $C_{a\cdot}$ and $C_{\cdot b}$, $a$, $b=0,1$.

For a power bounded operator $T\in\mathcal L(\mathcal H)$ 
\begin{equation*}\text{ the \emph{isometric asymptote} }(X_{+,T},T_+^{(a)}) \end{equation*}   
 can be defined using a Banach limit $\mathop{\text{\rm Lim}}$, see \cite{ker89}. (For the isometric and unitary asymptotes of a contraction $T$ see also {\cite[Sec. IX.1]{nfbk}}). 
Here $T_+^{(a)}$ is an isometry on a Hilbert space $\mathcal H_+^{(a)}$, and $X_{+,T}$ is the \emph{canonical intertwining mapping}: 
$X_{+,T}T=T_+^{(a)}X_{+,T}$. Recall that 
the range of $X_{+,T}$ is dense. 
Thus, $X_{+,T}$ realizes the relation $T\buildrel d\over\prec  T_+^{(a)}$.
 We do not recall the construction of the canonical intertwining mapping from \cite{ker89} here. 
We recall only that $\|X_{+,T}x\|^2=\mathop{\text{\rm Lim}}_n\|T^nx\|^2$ for every $x\in \mathcal H$. 
It easy follows from this relation that an operator  $T\in\mathcal L(\mathcal H)$ is similar 
to an isometry if and only if $T$ is power bounded and there exists $c>0$ such that 
$\|T^nx\|\geq c\|x\|$ for every $x\in\mathcal H$ and $n\in\mathbb N$. In this case, $X_{+,T}$ is invertible 
and realizes the relation $T\approx T_+^{(a)}$.

The \emph{unitary asymptote} $(X_T,T^{(a)})$ of a power bounded operator $T\in\mathcal L(\mathcal H)$ is a pair where 
 $T^{(a)}\in\mathcal L(\mathcal H^{(a)})$   
(here $\mathcal H^{(a)}$ is a some Hilbert space) 
is the minimal unitary extension of $T_+^{(a)}$, and $X_T$ is a natural extension of $X_{+,T}$. 
\emph{The isometry $T_+^{(a)}$ and the unitary operator  $T^{(a)}$  will also be called the isometric and unitary asymptotes of $T$, respectively.}

Let $S$ be the simple unilateral shift, 
that is, the multiplication by the independent variable on the Hardy space $H^2$ on the unit circle $\mathbb T$. 
A particular case of \cite{ker07} is the following (see also {\cite[Sec. IX.3]{nfbk}}). 
 Let $T\in\mathcal L(\mathcal H)$ be an absolutely continuous  (a.c.) contraction (the definition is recalled in Sec. 2 of the present paper), 
and let $T^{(a)}$ contain the bilateral shift as an orthogonal summand.
Then 
\begin{equation}\label{shifttype}\mathcal H=\vee\{\mathcal M \ : \ \mathcal M\in\operatorname{Lat}T,\  T|_{\mathcal M} \approx S\}. 
\end{equation}
In \cite{gamal22} this result is generalized  to a.c. polynomially bounded operators (the definition  
can be found, for example, in {\cite[Ch. 15]{peller}}, see also {\cite[Ch. I.13]{nfbk}} where other terminology is used; see references therein). 
Also it is shown in \cite{gamal22} that the quantity of subspaces $\mathcal M$ in \eqref{shifttype} can be equal to $\mu_T$, 
if $\mu_T\geq 2$, and to $2$, if $\mu_T=1$. On the other hand, there exists  power bounded operator $T$ such that  $T^{(a)}_+=S$ and 
there is no $\mathcal M\in\operatorname{Lat}T$ such that $T|_{\mathcal M} \approx S$ {\cite[Sec. 5]{gamal16}}.  
The purpose of this paper is  to show that \eqref{shifttype} is fulfilled for expansive operators $T$ which are quasisimilar 
to the unilateral shift $S_N=\oplus_{j=1}^N S$ of finite multiplicity $N\in\mathbb N$, and  
the quantity of subspaces $\mathcal M$ in \eqref{shifttype} is as decribed above (Theorems \ref{thm1} and \ref{thmnn}). Expansive operators are right inverses for contractions. 
The proof is based on the result for contractions from \cite{ker07} (see also {\cite[Sec. IX.3]{nfbk}})  and 
on representations of unimodular  functions on $\mathbb T$ given in \cite{bourgain} and developed in \cite{hjelle}. 
Some other properties 
 of an expansive operator $T$ such that $T\sim S_N$,  where $N\in\mathbb N$, are studied. In particular, it is proved that 
$\dim(\mathcal M\ominus T\mathcal M)\leq N$ for every $\mathcal M\in \operatorname{Lat}T$, and 
$I-T^*T\in \frak S_1$, where
 $\frak S_1$ is the trace class operators (Theorems \ref{thmmain0} and \ref{thmmain1}).  

The paper is organized as follows. In Sec. 2 some simple observations are collected, some of them are of own interest, and some of them will be used in the sequel. In Sec. 3 a special kind of finite perturbations of $S_N$ ($N\in\mathbb N$) that are  expansive operators  is considered. Sec. 4 is the main part of the paper. In Sec. 5 the relationship between similarity to an isometry 
of expansive operator and its Cauchy dual (adjoint of the standard left inverse) is studied. In Sec.  6 it is shown that there exist expansive operators $T$ such that $T\sim S$, but 
the intertwining quasiffinities do not give the isomorphism of $\operatorname{Lat}T$ and $\operatorname{Lat}S$ 
(in contrast with the case when $T$ is a contraction).

The following notation will be used. For a (closed) subspace $\mathcal M$ of a Hilbert space $\mathcal H$, by 
$P_{\mathcal M}$ and $I_{\mathcal M}$ the orthogonal projection from $\mathcal H$ onto $\mathcal M$ and 
 the identity operator on $\mathcal M$ are denoted,  respectively. By $\mathbb O$ the zero transformation acting between (maybe nonzero) spaces is denoted.

Symbols $\mathbb D$ and $\mathbb T$ denote the open unit disc
and the unit circle, respectively. The normalized Lebesgue measure on $\mathbb T$ is denoted by $m$. 
Set $L^p=L^p(\mathbb T,m)$. For $0<p\leq\infty$ by $H^p$ the Hardy space on $\mathbb T$ is denoted. Set $\chi(\zeta)=\zeta$ and $\mathbf{1}(\zeta)=1$ for $\zeta\in\mathbb T$. 
The simple unilateral  $S$ is the operator  of multiplication by $\chi$   on $H^2$.  Set $H^2_-=L^2\ominus H^2$. 
For a measurable set $\sigma\subset\mathbb T$ denote by $U_\sigma$ the operator of multiplication by $\chi$ on 
 $L^2(\sigma,m)$. Then $U_{\mathbb T}$ is the simple bilateral shift. 

For $N\in\mathbb N\cup\{\infty\}$ denote by  $H^2_N$, $L^2_N$, $(H^2_-)_N$   the orthogonal sum of $N$ copies of $H^2$, $L^2$, $H^2_-$, respectively. For $N\in\mathbb N$, vectors from $H^2_N$, $L^2_N$, $(H^2_-)_N$ are columns of functions from $H^2$, $L^2$, $H^2_-$, respectively. For $1\leq k\leq N$ denote by $e_k$ the vector from $H^2_N$ 
with $\mathbf{1}$ on $k$-th place and zeros on all other places.
Then $\{e_k\}_{k=1}^{N}$ is an orthonormal basis of $\ker S_N^*$. 
By $P_+$ and $P_-$ the orthogonal projections from $L^2_N$ onto $H^2_N$ and $(H^2_-)_N$  
are denoted, respectively (they depend on $N$, but it will not be mentioned in notation). 
Set $S_*=P_-U_{\mathbb T}|_{H^2_-}$. By $S_N$, $S_{*, N}$,  and $U_{\mathbb T, N}$
 the orthogonal sum of $N$ copies of $S$, $S_*$, and $U_{\mathbb T}$ are denoted,  respectively. 
Recall that $\mu_{S_N}=\mu_{U_{\mathbb T, N}}=N$, and $\mu_{U_{\mathbb T, N}|_{\mathcal M}}\leq N$ for every 
$\mathcal M\in\operatorname{Lat}U_{\mathbb T, N}$.

For a matrix 
$F=[f_{jk}]_{j,k}$ whose elemets are functions $f_{jk}$ set  $\overline F=[\overline f_{jk}]_{j,k}$.

\section{General observations}

The following lemma is well known and can be proved easily, so its proof is omitted. 

\begin{lemma}\label{lemdefect} Let $A$, $B\in\mathcal L(\mathcal H)$ be such that $BA=I_{\mathcal H}$ and $\dim\ker A^*<\infty$. 
Then the following are equivalent: $\mathrm{(i)}$ $ I_{\mathcal H}-A^*A\in\frak S_1$; $\mathrm{(ii)}$ 
 $I_{\mathcal H}-AA^*\in\frak S_1$; $\mathrm{(iii)}$ $I_{\mathcal H}-B^*B\in\frak S_1$;
$\mathrm{(iv)}$  $I_{\mathcal H}-BB^*\in\frak S_1$. 
\end{lemma}

Recall that  $A\in L(\mathcal H)$ is called a Fredholm operator, if $A\mathcal H$ is closed, $\dim\ker A<\infty$, and 
$\dim\ker A^*<\infty$.   
 Denote by $\operatorname{ind}$ the Fredholm index of a Fredholm operator $A$,  that is, 
$\operatorname{ind}A=\dim\ker A-\dim\ker A^*$. See, for example, {\cite[Ch. XI]{conway}}.

\begin{lemma} \label{lemfred} Suppose that $N\in\mathbb N$, $A\in L(\mathcal H)$, $\ker A=\{0\}$, $\dim\ker A^*=N$, 
and $Y\in\mathcal I(S_N,A)$ is such that $\operatorname{clos}YH^2_N=\mathcal H$.
Then $\ker Y=\{0\}$.
\end{lemma}

\begin{proof} We have 
 \begin{equation*}S_N=\left[\begin{matrix}S_N|_{\ker Y} & * \\ \mathbb O & R\end{matrix}\right], 
\end{equation*}
and $Y|_{H^2_N\ominus \ker Y}$ realizes the relation $R\prec A$. This relation implies that $\ker R=\{0\}$ and 
$\dim\ker R^*\geq N$. By {\cite[Theorem  XI.3.7]{conway}},  
 \begin{equation*} -N=\operatorname{ind}S_N=\operatorname{ind}S_N|_{\ker Y}+\operatorname{ind}R\leq
\operatorname{ind}S_N|_{\ker Y}-N.\end{equation*}
This means that $\operatorname{ind}S_N|_{\ker Y} =0$. Therefore, $\ker Y=\{0\}$.
\end{proof}

For $A\in\mathcal L(\mathcal H)$ set $\mathcal R^\infty(A)=\cap_{n\in\mathbb N}A^n\mathcal H$. 
If $\mathcal R^\infty(A)=\{0\}$, then $A$ is called \emph{analytic} \cite{shimorin} or \emph{pure} \cite{olofsson}. 
The following simple lemma is given for convenience of references; its proof is evident and  omitted.

\begin{lemma}\label{lemaabbxx}  Let $A$ and  $B$ be  operators, and let $X\in\mathcal I(A,B)$. 
Then $X\mathcal R^\infty(A)\subset\mathcal R^\infty(B)$.
\end{lemma}

Let $A$ be left-invertible, equivalently, let $A$ be bounded below:
 there exists $c>0$ such that $\|Ax\|\geq c\|x\|$ for every $x\in \mathcal H$. 
Then $\mathcal R^\infty(A)\in\operatorname{Lat}A$, $A|_{\mathcal R^\infty(A)}$ is invertible,  and if 
 $\mathcal M\in\operatorname{Lat}A$ is such that $A\mathcal M=\mathcal M$, then $\mathcal M\subset\mathcal R^\infty(A)$. 
Consequently, 
$P_{\mathcal H\ominus\mathcal R^\infty(A)}A|_{\mathcal H\ominus\mathcal R^\infty(A)}$ is left-invertible, and 
\begin{equation} \label{pure} 
\mathcal R^\infty(P_{\mathcal H\ominus\mathcal R^\infty(A)}A|_{\mathcal H\ominus\mathcal R^\infty(A)})=\{0\}.
\end{equation}  
For a left-invertible $A\in\mathcal L(\mathcal H)$ the operator $L_A=(A^*A)^{-1}A^*$ is the standard  left inverse for $A$: 
 $L_AA=I_{\mathcal H}$,  and $\ker L_A=\ker A^*$. Set $A'=L_A^*=A(A^*A)^{-1}$. The operator $A'$ is called the \emph{Cauchy dual} to $A$ (\cite{shimorin},  \cite{olofsson}).   Note that $A'$ is left-invertible and $A''=A$.

\begin{lemma} \label{lemshimorin} \begin{enumerate}[\upshape (i)]

\item {\cite[Prop. 2.7]{shimorin}}
 Let $A$, $B\in\mathcal L(\mathcal H)$ be such that $BA=I_{\mathcal H}$. Then 
$ \mathcal H=\mathcal R^\infty(A)\oplus\vee_{n=0}^\infty B^{*n}\ker A^*$.

\item {\cite[Lemma 2.1]{shimorin}}
 Let $A\in\mathcal L(\mathcal H)$ be left-invertible.
 \begin{equation*} \text { Let }\mathcal H=\vee_{n=0}^\infty A^n\ker A^*. \text{ Then  }
\mathcal H=\vee_{n=0}^\infty \ker L_A^n.\end{equation*}
\end{enumerate}
\end{lemma}

\begin{lemma}\label{lemexp} Let $T\in\mathcal L(\mathcal H)$ be expansive. 
Then $P_{\mathcal H\ominus\mathcal R^\infty(T)}T|_{\mathcal H\ominus\mathcal R^\infty(T)}$ is expansive.
\end{lemma}
\begin{proof} Let $x\in \mathcal H\ominus\mathcal R^\infty(T)$. Since $T\mathcal R^\infty(T)=\mathcal R^\infty(T)$, 
there exists $y\in\mathcal R^\infty(T)$ such that $Ty=P_{\mathcal R^\infty(T)}Tx$. We have
\begin{equation*} \|P_{\mathcal H\ominus\mathcal R^\infty(T)}Tx\|^2=\|T(x-y)\|^2\geq \|x-y\|^2=
\|x\|^2+\|y\|^2\geq \|x\|^2.\qedhere\end{equation*} 
\end{proof} 

Let $R$ be a contraction. Then $R=U_s\oplus U_a\oplus R_1$, where $U_s$ and $U_a$ are singular and absolutely 
continuous  unitary operators (that is, their spectral measures are singular and absolutely 
continuous  with respect to $m$), respectively, and $R_1$ is a completely nonunitary contraction (that  is, there is no 
$\{0\}\neq\mathcal M\in\operatorname{Lat}R_1$ such that  $T|_{\mathcal M}$ is unitary). 
If $U_s$ acts on the zero space $\{0\}$, then $R$ is called an \emph{absolutely continuous (a.c.)} contraction. If 
$U$ is a singular unitary operator and $R$ is an a.c. contraction, then $\mathcal I(R,U)=\mathbb O$. 
For an a.c. contraction $R$ the $H^\infty$-functional calculus is defined. If there exists $0\not\equiv\varphi\in H^\infty$ 
such that $\varphi(R)=\mathbb O$, then $R$ is called a \emph{$C_0$-contraction}. $C_0$-contractions are of class $C_{00}$. 
 For references, see {\cite[Theorems I.3.2, II.2.3, II.6.4, and Secs. III.2, III.4]{nfbk}}.

\begin{lemma}\label{lemnotuu} Let $T\in\mathcal L(\mathcal H)$ be expansive. Then $T'$ is a contraction.
Furthermore, the following statements hold true.
\begin{enumerate}[\upshape (i)]
\item Suppose that $U$ is a singular unitary operator,  $\mathcal M\in\operatorname{Lat}T$,  
and $T|_{\mathcal M}\approx U$. 
Then $T|_{\mathcal M}\cong U$ and  $\mathcal H\ominus\mathcal M\in\operatorname{Lat}T$. Also  
 $\mathcal M$,  $\mathcal H\ominus\mathcal M\in\operatorname{Lat}T'$ and  $T'|_{\mathcal M}\cong U$. 
\item If $R$ is an a.c. contraction such that $R\prec T$, then $T'$ is an a.c. contraction.
\item If $\mathcal R^\infty(T)=\{0\}$, 
then $T'$ is a completely non-unitary contraction; 
\item If  $\mathcal H=\vee_{n=0}^\infty T^n\ker T^*$, then $T'$ is a contraction of class $C_{\cdot 0}$.
\end{enumerate}
\end{lemma} 

\begin{proof} The estimate $\|T'\|\leq 1$ easy follows from the relations $T'^*T=I$ and $\ker T'^*=\ker T^*$. 

(i) Since $T'^*T=I$ and $T|_{\mathcal M}$ is invertible, we have 
$\mathcal M\in\operatorname{Lat}T'^*$ and  $T'^*|_{\mathcal M}=(T|_{\mathcal M})^{-1}\approx U^{-1}$. 
Since $T'^*$ is a contraction and $U$ is a singular unitary operator, we have 
$T'^*|_{\mathcal M}\cong U^{-1}$ and $\mathcal H\ominus\mathcal M\in\operatorname{Lat}T'^*$. 
The conclusion of part (i) of the lemma follows from these relations. 

 (ii) Assume that $T'$ is not an a.c. contraction.
 Therefore, there exist a singular unitary operator  $U$ and  $\{0\}\neq\mathcal M\in\operatorname{Lat}T'$ such that 
 $\mathcal H\ominus\mathcal M\in\operatorname{Lat}T'$ and  $T'|_{\mathcal M}\cong U$. 
Consequently, $\mathcal M$, $\mathcal H\ominus\mathcal M\in\operatorname{Lat}T$ and  $T|_{\mathcal M}\cong U$.
Let $Y$ be a quasiaffinity such that $YR=TY$.  
The transformation   
$P_{\mathcal M}Y$ realizes the relation $R\buildrel d \over\prec T|_{\mathcal M}$. Thus, $R\buildrel d \over\prec U$,
a contradiction.

(iii) Assume that  there exists 
$\mathcal K\in\operatorname{Lat}T'$ such that $U:=T'|_{\mathcal K}$ is unitary. Then $T'=U\oplus R$ for some  
 $R\in\mathcal L(\mathcal H\ominus\mathcal K)$. 
We have 
\begin{gather*} T'^*T'=I_{\mathcal K}\oplus R^*R, 
\ \ (T'^*T')^{-1}=I_{\mathcal K}\oplus (R^*R)^{-1},\\ \text{ and } 
T=T''=U\oplus R(R^*R)^{-1}. \end{gather*} 
 Consequently, $\mathcal K\subset\mathcal R^\infty(T)$. Thus, $\mathcal K=\{0\}$. 

(iv) This is a straightforward corollary of Lemma \ref{lemshimorin}(ii), because $T'=L_T^*$. 
\end{proof}

\begin{lemma}\label{lemmodel}
 Let $R\in\mathcal L(\mathcal H)$ be a contraction, 
 and let $1\leq N=\dim\ker R^*\leq \infty$. 
Then there exists $Y\in\mathcal I(S_N,R)$ such that  $Y\ker S_N^*=\ker R^*$ and 
$\operatorname{clos}Y H^2_N=\vee_{n=0}^\infty R^n\ker R^*$. Furthermore, if $R$ is left-invertible, then 
there exists $X\in\mathcal I(R',S_N)$ such that  $\operatorname{clos}X\mathcal H=H^2_N$, $X\ker R'^*=\ker S_N^*$ and 
$\ker X=\mathcal R^\infty(R')$. 
\end{lemma}
\begin{proof} By {\cite[Theorem I.4.1]{nfbk}}, there exists a Hilbert space $\mathcal K$ and an
 isometry $V\in\mathcal L(\mathcal K)$ such that $\mathcal H\subset\mathcal K$, 
$\mathcal K\ominus\mathcal H\in\operatorname{Lat}V$, and $R=P_{\mathcal H}V|_{\mathcal H}$. 
Set $E=\ker R^*$. Then $E=\ker V^*\cap\mathcal H$. 
Set   $\mathcal M=\oplus_{n=0}^\infty V^n E$. Then $\ker (V|_{\mathcal M})^*=E$ and $V|_{\mathcal M}\cong S_N$.
Set $Y=P_{\mathcal H}|_{\mathcal M}$. Then $Y$ satisfies the conclusion of the lemma.

Set $X=P_{\mathcal M}|_{\mathcal H}$. If $R$ is left-invertible,
 then $R'=P_{\mathcal H}V|_{\mathcal H}(V^*|_{\mathcal H}P_{\mathcal H}V|_{\mathcal H})^{-1}$ and 
$\ker R'^*=\ker R^*=E$. Let $x\in\mathcal H$, $v\in\mathcal M$ and $u\in E$. 
Then $(R'x,u)=0$ and $V^*(Vv+u)=v$. Therefore,
\begin{align*} (XR'x, Vv+u)&=(P_{\mathcal M}R'x,Vv+u)=(R'x,Vv+u)=(V^*R'x,v)\\&=
(V^*P_{\mathcal H}V|_{\mathcal H}(V^*|_{\mathcal H}P_{\mathcal H}V|_{\mathcal H})^{-1}x,v)=(x,v)=
(P_{\mathcal M}x,v)\\&=(Xx,v)=(Xx, V^*(Vv+u))=(VXx, Vv+u).\end{align*}
Since $\mathcal M=V\mathcal M\oplus E$, we conclude that $XR'=VX$. Clearly, $XE=E$. Since 
$E\subset\operatorname{clos}X\mathcal H\in\operatorname{Lat}V$, we have 
$\operatorname{clos}X\mathcal H=\mathcal M$. 

Set $\mathcal F=\ker X$. By Lemma \ref{lemaabbxx}, $\mathcal R^\infty(R')\subset \mathcal F$. Also, $\mathcal F\in\operatorname{Lat}R'$. 
Since $\mathcal F=\ker P_{\mathcal M}|_{\mathcal H}=\mathcal M^\perp\cap\mathcal H$, we have 
$\mathcal F\in\operatorname{Lat}V^*$. Consequently,  $\mathcal F\in\operatorname{Lat}R^*$. 
The equality $R^*R'=I_{\mathcal H}$ implies that $R^*|_{\mathcal F}R'|_{\mathcal F}=I_{\mathcal F}$. Therefore, 
 $R^*\mathcal F=\mathcal F$. Furthermore, 
\begin{equation*}
\ker R^*|_{\mathcal F}=E\cap\mathcal F\subset E\cap\mathcal M^\perp=\{0\}.
\end{equation*}
Thus, $R^*|_{\mathcal F}$ is invertible, and $(R^*|_{\mathcal F})^{-1}=R'|_{\mathcal F}$. Thus,  
$\mathcal F\subset\mathcal R^\infty(R')$. 
\end{proof}

 \begin{corollary}\label{cormodel} Suppose that $T$ is  expansive, $1\leq N=\dim\ker T^*\leq \infty$, 
and $\mathcal R^\infty(T)=\{0\}$. Then there exists a quasiaffinity $X\in\mathcal I(T,S_N)$ such that  $X\ker T^*=\ker S_N^*$.
\end{corollary}
\begin{proof} Set $R=T'$ and apply Lemma \ref{lemmodel} to $R$.\end{proof}

\begin{lemma}\label{lemyy} Let $A\in\mathcal L(\mathcal H)$ and  $B\in\mathcal L(\mathcal K)$ be power bounded operators, and let $Y\in\mathcal L(\mathcal H,\mathcal K)$ be 
such that $BYA=Y$. Then 
 \begin{equation}\label{asymp01} \mathcal H_{A,0}\subset\ker Y, \end{equation}
where $\mathcal H_{A,0}$ is defined by \eqref{01}. Consequently, 
\begin{enumerate}[\upshape (i)]
\item if $A$ is of class $C_{0\cdot}$, then $Y=\mathbb O$;
\item if $\ker Y=\{0\}$, then $A$ is of class $C_{1\cdot}$.
\end{enumerate}
Moreover, 
there exists $X_+\in\mathcal L(\mathcal H_+^{(a)},\mathcal K)$  such that
$\|X_+\|\leq\sup_{n\in\mathbb N}\|B^n\|\|Y\|$, 
$X_+X_{+,A}=Y$ and $X_+=BX_+A_+^{(a)}$.
\end{lemma}
\begin{proof} For every $n\in\mathbb N$ we have $B^nYA^n=Y$. Set $C=\sup_{n\in\mathbb N}\|B^n\|$.
 Then $\|Yx\|\leq C\|Y\|\|A^nx\|$ for every $x\in\mathcal H$ and every $n\in\mathbb N$. Consequently, \eqref{asymp01} is fulfilled. 
 
Set $X_+(X_{+,A}x)=Yx$ for $x\in\mathcal H$. Inclusion \eqref{asymp01} implies that the definition is correct.    We have
\begin{equation*} \|X_+X_{+,A}x\|^2= \|Yx\|^2\leq C^2\|Y\|^2\mathop{\mathrm{Lim}}_n\|A^nx\|^2=C^2\|Y\|^2\|X_{+,A}x\|^2,  \end{equation*}
where $\mathop{\mathrm{Lim}}$ is a Banach limit which is used in the construction of $(X_{+,A},A_+^{(a)})$, 
and
\begin{equation*}X_+X_{+,A}=Y=BYA=BX_+X_{+,A}A=BX_+A_+^{(a)}X_{+,A}.
\end{equation*}
Since the range of $X_{+,A}$ is dense, we conclude that 
$X_+$ can be extended as (linear, bounded) transformation onto $\mathcal H_+^{(a)}$, and $X_+=BX_+A_+^{(a)}$. 
\end{proof}

\begin{corollary}\label{lemasymp11} Suppose that $T\in\mathcal  L(\mathcal H)$ is an expansive operator,  
$R\mathcal\in \mathcal L(\mathcal K)$ is a power bounded operator, and $Z\in\mathcal I(R,T)$. Then 
  $\mathcal K_{R,0}\subset\ker Z$, where  $\mathcal K_{R,0}$ is defined by \eqref{01}. Moreover, 
there exists $Y\in\mathcal I(R_+^{(a)},T)$  such that $\|Y\|\leq\|Z\|$  and $Z=YX_{+,R}$.
\end{corollary}
\begin{proof} Since $ZR=TZ$ and $T'^*T=I$, we have $T'^*ZR=Z$. By Lemma \ref{lemyy}, there exists 
$Y\in\mathcal L(\mathcal K_+^{(a)},\mathcal H)$  such that $Z=YX_{+,R}$, and  $\|Y\|\leq\|Z\|$, 
because $\|T'\|\leq 1$. Furthermore, 
 \begin{equation*}  TYX_{+,R} = TZ = ZR = YX_{+,R} R = YR_+^{(a)} X_{+,R}. 
\end{equation*}
Since the range of $X_{+,R}$ is dense, we conclude that $TY=YR_+^{(a)}$. 
\end{proof}

\begin{corollary}\label{colssquasi}  Suppose that $N\in\mathbb N$, $T$ is expansive, $R$ is a contraction, and 
$R\prec T\prec S_N$. Then $T\sim S_N$.
\end{corollary}
\begin{proof} Since $R\prec S_N$, we have $R_+^{(a)}\cong S_N$  by {\cite[Lemma 2.1]{gamalshiftindex}}.
 Denote by $Z$ and $X$ the quasiaffinities such that $ZR=TZ$ and $XT=S_NX$.  By Corollary \ref{lemasymp11}, there exists $Y\in\mathcal I(S_N,T)$  such that 
$Z=YX_{+,R}$. The latest equality implies that $Y$ has dense range. 
 Applying  Lemma \ref{lemfred}  to $XY$ with $A=S_N$, we obtain that $\ker XY=\{0\}$. Consequently, 
$\ker Y=\{0\}$.
\end{proof}

\begin{corollary} Suppose that $N\in\mathbb N$, $T$ is expansive, and $S_N\buildrel d\over\prec T$. 
Then there exists $M\in\mathbb N$ such that $M\leq N$ and $S_M\prec T$.
\end{corollary}

\begin{proof} Let $Y_0$ realize the relation $S_N\buildrel d\over\prec T$. Set
\begin{equation*}\mathcal N=\ker Y_0, \ \ \ R=P_{H^2_N\ominus \mathcal N}S_N|_{H^2_N\ominus \mathcal N}
\ \ \text{and } \ \ Z=Y_0|_{H^2_N\ominus \mathcal N}. \end{equation*} 
Then $Z$ realizes the relation $R\prec T$. By Corollary \ref{lemasymp11}, $R\in C_{1\cdot}$. By 
\cite{uchiyama83} or \cite{takahashi84}, and \cite{ker89} or {\cite[Sec. IX.1]{nfbk}}, and
{\cite[Lemma 2.1]{gamalshiftindex}}, there exists $M\leq N$ such that  $R_+^{(a)}\cong S_M$. By Corollary \ref{lemasymp11}, 
 there exists $Y\in\mathcal I(S_M, T)$ such that $Z=YX_{+,R}$. The range of $Y$ is dense, because the range of $Z$ is dense. 
Set $\mathcal E=\ker Y$. Then $\mathcal E\in\operatorname{Lat}S_M$. By \cite{gamal02}, there exists 
$\mathcal M\in\operatorname{Lat}R$ such that $\mathcal E=\operatorname{clos}X_{+,R}\mathcal M$.
Consequently, $\mathcal M\subset\ker Z=\{0\}$. Thus, $\mathcal E=\{0\}$.
\end{proof}

\section{Intertwining by Toeplitz operators}

Let $\theta\in H^\infty$ be an inner function. Set $\mathcal K_\theta=H^2\ominus\theta H^2$. Then 
\begin{equation}\label{30}\mathcal K_\theta=\theta\overline\chi\overline{\mathcal K_\theta}=P_+\theta H^2_-,
\end{equation}
and $\mathcal K_\theta\in\operatorname{Lat}S^*$. 
Set $S(\theta)=P_{\mathcal K_\theta}S|_{\mathcal K_\theta}$.
Then 
\begin{equation}
\label{sstheta}
S(\theta) f =\chi f-(f,P_+\overline\chi\theta)\theta \ \ ( f\in\mathcal K_\theta). 
\end{equation}

For an inner function $\theta$ such that $\theta(0)=0$ 
there exists a singular (with respect to $m$) positive Borel  measure $\nu$ on $\mathbb T$ such that $\nu(\mathbb T)=1$ and 
\begin{equation}
\label{thetaclark}
\frac{1}{1-\theta(z)}=\int_{\mathbb T}\frac{1}{1-z\overline\zeta}\mathrm{d}\nu(\zeta)  \ \ (z\in\mathbb D),\end{equation}
which is called the \emph{Clark measure} of $\theta$. 
Every function $f\in\mathcal K_\theta$ has nontangential boundary values $f(\zeta)$ for 
a.e. $\zeta\in\mathbb T$ with respect to $\nu$, and 
\begin{equation}
\label{yy1}
 f(z)=(1-\theta(z))\int_{\mathbb T}\frac{f(\zeta)}{1-z\overline\zeta}\mathrm{d}\nu(\zeta)  \ \ (z\in\mathbb D).
\end{equation}
The relation \eqref{yy1} and F. and M. Riesz theorem (see, for example, {\cite[Theorem 1.21]{gmr}}
 or {\cite[Theorem II.3.8]{garnett}}) imply that 
\begin{equation}\label{31} (1-\theta)H^2\cap\mathcal K_\theta=\{0\}. \end{equation}
For an inner function $\theta$ such that $\theta(0)=0$ set 
\begin{equation}
\label{uutheta}
U(\theta)=S(\theta)+\mathbf{1}\otimes\overline\chi\theta.
\end{equation}
For a  singular positive Borel  measure $\nu$ on $\mathbb T$ such that $\nu(\mathbb T)=1$ denote by $U_\nu$ 
the operator of multiplication by the independent variable on $L^2(\mathbb T, \nu)$. If $\nu$ is 
the Clark measure for $\theta$, then 
\begin{equation}
\label{thetanucong}
U(\theta)\cong U_\nu.
\end{equation}

Conversely, if $\nu$ is a singular positive Borel  measure on 
$\mathbb T$ such that $\nu(\mathbb T)=1$ and $\theta$ is defined by \eqref{thetaclark}, 
then $\theta$ is an inner function and $\theta(0)=0$. 

For references,  see \cite{clark} and \cite{polt} or \cite{garciaross15}, \cite{gmr}.

If $\theta$ is an inner function such that $\theta(0)\neq 0$, then $S(\theta)$ is invertible and 
it is follows from \eqref{sstheta} that 
\begin{equation}\label{ssthetainv} (S(\theta)^*)^{-1}=S(\theta)+(\theta-\frac{1}{\overline{\theta(0)}})\otimes P_+\overline\chi\theta.
\end{equation}

The \emph{Toeplitz operator} $T_\psi$ with  the symbol $\psi\in L^2$ acts by the formula 
$T_\psi h=P_+\psi h$ for $h\in H^\infty$. It can be extended as a \emph{bounded} operator on $H^2$   if and only if 
 $\psi\in L^\infty$, and then it acts by the formula $T_\psi h=P_+\psi h$ ($h\in H^2$). 
The following lemma can be found, for example, in {\cite[Theorem 3.1.2]{peller}}.

\begin{lemma}\label{lemcc} Let $T\in\mathcal L(H^2)$. Then $T=T_\psi$ for some $\psi\in L^\infty$ if and only if $S^*TS=T$.
\end{lemma}

 It can be checked by the straightforward calculation that 
\begin{equation}\label{toeplitz} T_\psi S-ST_\psi=\mathbf{1}\otimes P_+\overline{\chi\psi}.
\end{equation}

Let $0\not\equiv g\in H^2$, and let $f\in H^2$ be such that $|f|\leq |g|$ $m$-a.e. on $\mathbb T$. Using the equality $H^1\cap\overline\chi\overline H^1=\{0\}$ it is easy to see that 
$\ker T_{\frac{f}{\overline g}}=\{0\}$. If $g$, $1/g\in H^2$, then $ T_{\frac{g}{\overline g}}$ is a quasiaffinity. 
A description of functions $g$ such that $\ker T_{\frac{\overline g}{g}}=\{0\}$ is given in \cite{sarason89}. 
A necessary condition for $\ker T_{\frac{\overline g}{g}}=\{0\}$ is that $g$ is outer, and $g\neq(1-\theta)f$ for any inner function 
 $\theta\in H^\infty$ and any $f\in H^2$. But this condition is not sufficient \cite{inoue}. With respect to this subject see also \cite{po}.

For $g\in H^2$ such that $\|g\|=1$ define the function $\omega\in H^\infty$ as follows:

\begin{equation}\label{nakamura1} \frac{1}{1-\omega(z)}=\int_{\mathbb T}\frac{|g(\zeta)|^2\mathrm{d}m(\zeta)}{1-z\overline\zeta}
 \ \ \ (z\in\mathbb D).
\end{equation}
Then $1-\omega$ is an outer function. The following theorem is proved in \cite{sarason}. 

\begin{theoremcite}\label{theoremaa}{\cite[Lemma 2]{sarason}} Let $g\in H^2$, and let $\|g\|=1$. Define $\omega$ by 
\eqref{nakamura1}. Then  
$T_{1-\omega} T_{\overline g}\in\mathcal L(H^2)$.
\end{theoremcite}

A simplest example of expansive operator is one-dimensional perturbation of $S$ from Lemma \ref{lemgsimss} below. Other one-dimensional perturbations of $S$  was considered in \cite{nakamura93} and \cite{cassiertimotin}.

\begin{lemma} \label{lemgsimss} Let $g\in H^2$, and let $g(0)=1$. Set $T=S-\mathbf{1}\otimes S^*g$. Set $g=\theta f$, where $\theta$ is inner and $f$ is outer. Then the following are fulfilled. 
\begin{enumerate}[\upshape (i)]
\item  
 $\mathcal K_\theta=\mathcal R^\infty(T)$, and 
$T|_{\mathcal K_\theta}=(S(\theta)^*)^{-1}$. 
\item  
$T_{\frac{f}{\overline g}}S=TT_{\frac{f}{\overline g}}$, and  there exists an outer function $\varphi\in H^\infty$ such that
$T_\varphi T_{\overline g}\in\mathcal L(H^2)$, and $T_\varphi T_{\overline g}T=ST_\varphi T_{\overline g}$.
\item The following are equivalent: 

$\mathrm{(a)}$ $T$ is similar to an isometry;

$\mathrm{(b)}$ $T\approx S$;

$\mathrm{(c)}$  $T_{\frac{g}{\overline g}}$ is invertible.
\end{enumerate}
\end{lemma}

\begin{proof} Straightforward computation shows that 
 $\mathcal K_\theta\in\operatorname{Lat}T$  and $T|_{\mathcal K_\theta}=(S(\theta)^*)^{-1}$. 
Thus, $\mathcal K_\theta\subset\mathcal R^\infty(T)$. Existence of $\varphi$ from (ii) follows from Theorem
\ref{theoremaa}. Intertwining relation from (ii) can be checked by straightforward computation. 
Set $X= T_\varphi T_{\overline g}$. By Lemma \ref{lemaabbxx},   $\ker X\supset\mathcal R^\infty(T)$, 
because $XT=SX$. Since $\ker X=\mathcal K_\theta$, we have $\mathcal K_\theta\supset\mathcal R^\infty(T)$. 
Thus, (i) and (ii) are proved.

It follows from (i) that if $T$ is similar to isometry, then $g$ is outer. Indeed, if $T$ is similar to an isometry, then 
 $(S(\theta)^*)^{-1}\approx U$ for some unitary $U$. Consequently, $S(\theta)^*\approx U^{-1}$. 
But $S(\theta)$ is a $C_0$-contraction. Therefore, $\mathcal K_\theta=\{0\}$. Also, it is easy to see that (c) implies that $g$ is outer. 

Suppose that $g$ is outer. Then $X$ is a quasiaffinity. Therefore, (a)$\Rightarrow$(b). The relation (c)$\Rightarrow$(b) follows from (ii), and the relation (b)$\Rightarrow$(a) is evident.

Let $Y\in\mathcal L(H^2)$ be such that $YS=TY$. Then $S^*YS=Y$. By Lemma \ref{lemcc},  
 there exists $\psi\in L^\infty$ such that 
$Y=T_\psi$. The equality $T_\psi S-ST_\psi=- \mathbf{1}\otimes T_\psi^*S^*g$ and \eqref{toeplitz} imply that $\psi=\frac{h}{\overline g}$ 
for some $h\in H^2$. 

If $T_\psi$ is invertible, then by {\cite[Lemma 3.1.10]{peller}} $1/\psi\in L^\infty$. Therefore, $h=\vartheta\eta g$,  
where  $\vartheta$ is inner  and $\eta$, $1/\eta\in H^\infty$. Since $T_\psi=T_{\frac{\vartheta g}{\overline g}}T_\eta$ 
and $T_\eta$ is ivertible, we conclude that $T_{\frac{\vartheta g}{\overline g}}$ is invertible. If $\vartheta$
is not a constant, then $0\not\equiv gS^*\vartheta\in\ker T_{\frac{\vartheta g}{\overline g}}^*$. 
Thus, $\vartheta$ is a constant, and (b)$\Rightarrow$(c) is proved.
\end{proof}

\begin{example} Let $\theta\in H^\infty$ be an inner function, and let $\theta(0)=0$. 
Let $\nu$ be the Clark measure for $\theta$. 
Set \begin{equation*} T=S+\mathbf{1}\otimes\overline\chi\theta. \end{equation*} 
Then $\theta H^2\in\operatorname{Lat}T$, $T|_{\theta H^2}\cong S$ and 
$P_{\mathcal K_\theta}T|_{\mathcal K_\theta}\cong U_\nu$ (see \eqref{uutheta} and \eqref{thetanucong}).
Set $Y=T_\theta$ and $X=T_{1-\overline\theta}$. 
It is easy to see that $YS=TY$ and $XT=SX$. 
Since $X$ is a quasiaffinity, we have $T\prec S$. Since $\mathcal I(U_\nu^*,S^*)=\{0\}$, 
we conclude that $S\not\prec T$.
\end{example}

\begin{example} Let $\theta\in H^\infty$ be an inner function, and let $\theta(0)\neq 0$. 
Set \begin{equation*} T=S+\mathbf{1}\otimes\overline\chi(1-\theta/\theta(0)). \end{equation*} 
By Lemma \ref{lemgsimss}, $\mathcal R^\infty(T)=\mathcal K_\theta$ and $T|_{\mathcal K_\theta}=(S(\theta)^*)^{-1}$. 
Furthermore,  $\theta H^2\in\operatorname{Lat}T$ and $T|_{\theta H^2}\cong S$. 
 Let $X\in\mathcal I(T,S)$. By Lemma \ref{lemaabbxx}, 
$X|_{\mathcal K_\theta}=\mathbb O$.  Let $Y\in\mathcal I(S,T)$. Set $Y_0=P_{\mathcal K_\theta}Y$. Then $Y_0S=(S(\theta)^*)^{-1}Y_0$. 
Consequently, $Y_0^*=S^*Y_0^*S(\theta)$. By Lemma \ref{lemyy} (i), $Y_0=\mathbb O$. Thus, 
$S\not\prec T\not\prec S$.
\end{example}

In the proof of the next lemma, 
Toeplitz operators with matrix-valued symbols 
 are used. Namely, let $N\in\mathbb N$, and let $\Psi$ be an $N\times N$ matrix whose elements are functions from $L^2$. 
The \emph{Toeplitz operator} $T_\Psi$ with the symbol $\Psi$ acts by the formula $T_\Psi h=P_+\Psi h$ for $h\in H^2_N$ 
such that their elements are  functions  from $H^\infty$. It can be extended as a \emph{bounded} operator on $H^2_N$ if and only if all elements of the matrix $\Psi$ are functions from $L^\infty$, and then it acts by the formula 
 $T_\Psi h=P_+\Psi h$  ($h\in H^2_N$).  See, for example, {\cite[Sec. 3.4]{peller}}.

\begin{lemma}\label{lemma1} Let $N\in\mathbb N$, and let $\{f_k\}_{k=1}^{N}\subset H^2_N$. 
Set $T=S_N+\sum_{k=1}^{N}e_k\otimes f_k$. 
Then $S_N\buildrel i\over\prec T\buildrel d\over\prec S_N$.
\end{lemma} 

\begin{proof} The case $N=1$ is considered in Lemma \ref{lemgsimss}.  Consider the case $N\geq 2$.  
Recall that $f_k$ are columns of $N$  functions from $H^2$. 
Denote by $f_{kj}$ the element of $f_k$ on $j$-th place. 
Set \begin{equation*} F=[f_{kj}]_{k,j=1}^N=[f_1, \ldots, f_N] \ \ \text{ and } \ \ \psi=\det(I_{N\otimes N}-\chi F).
\end{equation*}
 Denote by $f_{\mathrm{Ad}kj}$ 
($k,j=1,\ldots,N$) the elements of (algebraic) adjoint matrix of $I_{N\otimes N}-\chi F$. 
 Since  $f_{kj}\in H^2$, we have $\psi\in H^{\frac {2}{N}}$ and 
 $f_{\mathrm{Ad}kj}\in H^{\frac {2}{N-1}}$ 
($k,j=1,\ldots,N$). 
Since $\psi(0)=1$, we have $\psi\not\equiv 0$. Therefore, 
$\log|\psi|$, $\log|f_{\mathrm{Ad}kj}|\in L^1$ ($k,j=1,\ldots,N$), 
and the elements of the matrix $(I_{N\otimes N}-\chi F)^{-1}$ are functions defined 
$m$-a.e. on $\mathbb T$.  Furthermore, there exists an outer function $\eta\in H^\infty$ such that 
\begin{equation*} |\eta|=\begin{cases} 1, &\text{ if } |f_{\mathrm{Ad}kj}|\leq |\psi| \text { for all } k,j=1,\ldots,N,\\
\displaystyle{ \frac{|\psi|}{\max_{k,j=1,\ldots,N}|f_{\mathrm{Ad}kj}|}}, &\text{ if } |f_{\mathrm{Ad}kj}|\geq |\psi| \text { for some } k,j=1,\ldots,N
\end{cases}  \end{equation*}
$m$-a.e. on $\mathbb T$.
Set $\Psi=\eta\bigl((I_{N\otimes N}-\overline\chi \overline F)^{-1}\bigr)^{\mathrm T}$. 
Then the elements $\psi_{kj}$ (where $k$ is the number of row and $j$ is the number of column ($k,j=1,\ldots,N$))
 of $\Psi$ are functions from $L^\infty$. Set $Y=T_\Psi$ and 
 $\psi_k=[P_+\overline\chi\overline\psi_{kj}]_{j=1}^N$ ($k=1,\ldots,N$). Then 
$YS_N-S_NY=\sum_{k=1}^{N}e_k\otimes \psi_k$. Since $\psi_k=Y^*f_k$ ($k=1,\ldots,N$), we have  $YS_N=TY$. 

Denote by $\varphi_{kj}$ ($k,j=1,\ldots,N$) 
the outer  functions from Theorem \ref{theoremaa} applied to the elements of 
$(I_{N\otimes N}-\overline\chi \overline F)^{\mathrm T}$ (which multiplied by appropriate constants). Set $\varphi=\prod_{1\leq k,j\leq N}\varphi_{kj}$. 
Set $X=T_{\varphi I_{N\otimes N}}T_{(I_{N\otimes N}-\overline\chi \overline F)^{\mathrm T}}$. Then 
$X\in\mathcal L(H^2_N)$.  
Straightforward calculation shows that $XT=S_NX$. 

If $g\in H^2$ and $\gamma\in L^\infty$, then $P_+\overline gP_+\gamma=P_+\overline g\gamma$. 
 Therefore, if  $h\in H^2_N$ is such that its  elements are functions from $H^\infty$, then  
\begin{equation*} XYh=\varphi P_+(I_{N\otimes N}-\overline\chi \overline F)^{\mathrm T}P_+\Psi h=
\varphi P_+(I_{N\otimes N}-\overline\chi \overline F)^{\mathrm T}\Psi h=\varphi\eta h.
 \end{equation*}
Since $X$ and $Y$ are bounded, we conclude that $XYh=\varphi\eta h$ ($h\in H^2_N$). 
Consequently,  $\ker XY=\{0\}$ and $\operatorname{clos}XYH^2_N=H^2_N$ (since $\varphi$ and $\eta$ are outer).
Therefore, $\ker Y=\{0\}$ and $\operatorname{clos}XH^2_N=H^2_N$.
\end{proof}

\section{Expansive operators for which the  unilateral shift of finite multiplicity is their quasiaffine transform}

\subsection{Preliminaries}

In this subsection, some relationships between isometries are studies, which will be used in the sequel. Also, Theorem \ref{theorembb} from \cite{hjelle} is formulated in the end of this subsection.

\begin{lemma}\label{lemvv} Let an isometry $V$ have the representation 
\begin{equation*} V=\left[\begin{matrix}V_1 & * \\ \mathbb O & V_0\end{matrix}\right],
\end{equation*}
where $V_0$ is of class $C_{00}$. Then $V\cong V_1$.
\end{lemma}

\begin{proof} Let $V_1=U\oplus S_N$ be the Wold decomposition of the isometry $V_1$, where $U$ is unitary and $0\leq N\leq\infty$. 
Then  \begin{equation*} V=U\oplus V_{10},\ \ \  \text{ where }\  V_{10}=\left[\begin{matrix}S_N & * \\ \mathbb O & V_0\end{matrix}\right].
\end{equation*}
Since $S_N$ and $V_0$ are of class $C_{\cdot 0}$, then $V_{10}$ is of class $C_{\cdot 0}$, too, by {\cite[Theorem 3]{ker89}} or {\cite[Theorem IX.1.6]{nfbk}}
(applied to adjoint). Since $V_0$ is of class $C_{0\cdot}$, by  {\cite[Theorem 3]{ker89}} or {\cite[Theorem IX.1.6]{nfbk}}, 
 $V_{10}^{(a)}\cong S_N^{(a)}=U_{\mathbb T, N}$. Since $V_{10}$ is an isometry, we conclude that $V_{10}\cong S_N$.
\end{proof}

\begin{lemma}\label{lemasymp1} Suppose that a power bounded operator $R$ has the form 
\begin{equation*} R=\left[\begin{matrix}R_1 & * \\ \mathbb O & R_0\end{matrix}\right],\end{equation*}
and there exists a $C_0$-contraction $A$ such that $A\buildrel d\over\prec R_0$.
Then $R_+^{(a)}\cong (R_1)_+^{(a)}$.
\end{lemma}
\begin{proof} Denote by $\mathcal K$ the space on which $R$ acts. Let 
$\mathcal K=\mathcal K_1\oplus\mathcal K_0$ be the  decomposition of $\mathcal K$ such that 
$R_1=R|_{\mathcal K_1}$ and $R_0=P_{\mathcal K_0}R|_{\mathcal K_0}$.  
Set $\mathcal G_1=\operatorname{clos}X_{+,R}\mathcal K_1$, 
$\mathcal G_0= \mathcal K_+^{(a)}\ominus\mathcal G_1$, and  $V=R_+^{(a)}$. 
Then
\begin{equation*} V=\left[\begin{matrix}V_1 & * \\ \mathbb O & V_0\end{matrix}\right]
\end{equation*}
  with respect to the decomposition 
$\mathcal K_+^{(a)}=\mathcal G_1\oplus\mathcal G_0$. By \cite{ker89}, 
 $X_{+, R_1}=X_{+, R}|_{\mathcal K_1}$ and $(R_1)_+^{(a)}=V|_{\mathcal G_1}=V_1$.
 
We have $A\buildrel d\over\prec R_0\buildrel d\over\prec V_0$. Since $A$ is a $C_0$-contraction, $V_0$ is a $C_0$-contraction, too. 
In particular, $V_0$ is of class $C_{00}$ {\cite[Prop. III.4.2]{nfbk}}. By Lemma \ref{lemvv}, $V\cong V_1$.
\end{proof}

\begin{lemma} \label{lem11}Suppose that $\sigma\subset\mathbb T$, $X\in\mathcal I(U_\sigma^{-1}, S^*)$, and there exists 
 $f_1\in L^2(\sigma,m)$   such that $Xf_1=\mathbf{1}$.
 Then $\sigma=\mathbb T$ 
and \begin{equation*} U_{\mathbb T}|_{\vee_{n=0}^\infty U_{\mathbb T}^n f_1}\cong S.\end{equation*}
\end{lemma}

\begin{proof}  We have $X^*S=U_\sigma X^*$. Set $X^*\mathbf{1}=\psi$, 
then $\psi \in L^\infty (\sigma,m)$ and $X^*h=\psi h$ for every $h\in H^2$. Therefore, 
$Xf=P_+\overline\psi f$ for every $f\in L^2(\sigma,m)$. Since $\mathbf{1}=P_+\overline\psi f_1$,   
there exists $h\in H^2$ such that $1+\overline\chi \overline h=\overline\psi f_1$ $m$-a.e. on $\mathbb T$. 
Since $\psi=0$ $m$-a.e. on $\mathbb T\setminus\sigma$ and $1+\chi h\in H^2$, we conclude that 
$m(\mathbb T\setminus\sigma)=0$. Furthermore, 
\begin{equation*} \int_{\mathbb T}\log(|\psi||f_1|)\mathrm{d}m=
\int_{\mathbb T}\log|1+\chi h|\mathrm{d}m>-\infty.
\end{equation*}
Since $\psi \in L^\infty$, we conclude that $\int_{\mathbb T}\log|f_1|\mathrm{d}m>-\infty$. 
The conclusion of the lemma follows from this relation and well-known description of $\operatorname{Lat}U_{\mathbb T}$. 
\end{proof}

\begin{lemma} \label{lem12} Suppose that $N\in\mathbb N$,
  $V_+\in\mathcal L(\mathcal K_+)$ is an isometry, $\dim\ker V_+^*<\infty$, 
 $X_+\in\mathcal L(\mathcal K_+, H_N^2)$, and $S_N^*X_+V_+=X_+$.
 Let $V\in\mathcal L(\mathcal K)$ be the minimal unitary extension of $V_+$.
Then there exists $X\in\mathcal L(\mathcal K, H_N^2)$ such  $S_N^*XV=X$ and 
$X|_{\mathcal K_+}=X_+$.
\end{lemma}

\begin{proof}
Using the Wold decomposition and appropriate unitary equivalence, we may assume that  $V_+= S_M\oplus U$, 
where $U\in\mathcal L(\mathcal G)$ is unitary and $1\leq M\leq\dim\ker V_+^*$. Then $V=U_{\mathbb T,M}\oplus U$. 
Set $X_1=X_+|_{H_M^2\oplus\{0\}}$ and $X_0=X_+|_{\{0\}\oplus\mathcal G}$. Then $S_N^*X_1S_M=X_1$ and $S_N^*X_0U=X_0$. 
Writing $S_N^*$ and  $S_M$ as $N\times N$ and $M\times M$ diagonal   matrices, whose elements on the main diagonal are 
$S^*$ and $S$, respectively, and $X_1$ as a $N\times M$ matrix: $X_1=[X_{+jk}]_{j=1,\ldots, N\atop k=1,\ldots,M}$, 
we have $S^*X_{+jk}S=X_{+jk}$ for all $j=1,\ldots, N$, $k=1,\ldots,M$. By Lemma \ref{lemcc}, there exist 
$\psi_{jk}\in L^\infty$ such that $X_{+jk}=T_{\psi_{jk}}$.  Define $X_{jk}\in\mathcal L(L^2, H^2)$
by the formula $X_{jk}f=P_+\psi f$ ($f\in L^2$). 
Set \begin{equation*} X=\left[[X_{jk}]_{j=1,\ldots, N \atop k=1,\ldots,M},\ X_0\right].
\end{equation*}
It is easy to see that $X$ satisfies the conclusion of the lemma. 
\end{proof}

\begin{lemma}\label{lem13} Let $N\in\mathbb N$. Write $L^2_{N+1}=H^2_N\oplus (H^2_-)_N\oplus L^2$. 
Let $h_0\in H^2_N$, and let $f\in L^2$ be such that $\int_{\mathbb T}\log|f|\mathrm{d}m>-\infty$. 
Set \begin{equation*}
\mathcal M=H^2_N\vee\vee_{n=0}^\infty U_{\mathbb T, N+1}^n(\overline\chi\overline h_0\oplus f).
\end{equation*}
Then  $U_{\mathbb T, N+1}|_{\mathcal M}\cong S_{N+1}$.
\end{lemma}

\begin{proof} Set $\mathcal N=\vee_{n=0}^\infty( S_{*,N}^n\overline\chi\overline h_0\oplus U_{\mathbb T}^n f)$.
Then $\mathcal M=H^2_N\oplus\mathcal N$. We show that 
 \begin{equation}\label{nncap}
\mathcal N\cap((H^2_-)_N\oplus\{0\})=\{0\}. 
\end{equation}
Indeed, assume that $\{p_n\}_n$ is a sequence of analytic polynomials, $h\in H^2_N$,  
\begin{equation*} p_n(S_{*,N}) \overline\chi\overline h_0\to \overline\chi\overline h \ \ \text{ and } \ \ \ 
p_n( U_{\mathbb T})f \to 0.
\end{equation*}
Let $h_0=[h_j]_{j=1}^N$, where $h_j\in H^2$ ($j=1,\ldots, N$). Set $s(\zeta)=\max_{j=1,\ldots, N}|h_j(\zeta)|$ for $m$-a.e. $\zeta\in\mathbb T$.
Since $\int_{\mathbb T}\log|f|\mathrm{d}m>-\infty$, there exists an outer function $\varphi\in H^\infty$ such that 
\begin{equation*} |\varphi|=\begin{cases}\frac{|f|}{s}, & \ \ \text{ if }|f|\leq s,\\
1, & \ \ \text{ if }|f|\geq s.
\end{cases}
\end{equation*}
We have \begin{equation*}\varphi(S_{*,N})p_n(S_{*,N}) \overline\chi\overline h_0=[P_-\varphi P_-p_n\overline\chi\overline h_j]_{j=1}^N
=[P_-\varphi p_n\overline\chi\overline h_j]_{j=1}^N
\to \varphi(S_{*,N})\overline\chi\overline h.
\end{equation*}
But  \begin{align*}\|\varphi(S_{*,N})p_n(S_{*,N}) \overline\chi\overline h_0\|^2&
\leq\sum_{j=1}^N\|\varphi p_n\overline\chi\overline h_j\|^2
\leq \sum_{j=1}^N\int_{\mathbb T}|\varphi|^2s^2|p_n|^2\mathrm{d}m\\&\leq N\int_{\mathbb T}|f|^2|p_n|^2\mathrm{d}m\to 0.
\end{align*}
We obtain that $\varphi(S_{*,N})\overline\chi\overline h=0$. Since $\varphi$ is outer, {\cite[Prop. III.3.1]{nfbk}} implies that  $\overline\chi\overline h=0$. Thus, \eqref{nncap} is proved. 

Set $R= (S_{*,N}\oplus U_{\mathbb T})|_{\mathcal N}$. 
There exist $u\in L^\infty$ and $g\in H^2$ such that $|u|=1$ $m$-a.e. on $\mathbb T$, $g$ is outer, and $f=ug$. 
We have 
$\mathcal N\subset (H^2_-)_N\oplus uH^2$. By \eqref{nncap}, $P_{\{0\}\oplus uH^2}|_{\mathcal N}$ realizes the relation $R\prec U_{\mathbb T}|_{uH^2}$. 
Since $U_{\mathbb T}|_{uH^2}\cong S$ and $R$ is a contraction, we have $\operatorname{ind}R=-1$ \cite{takahashi}.
Since   
\begin{equation*} U_{\mathbb T, N+1}|_{\mathcal M}=\left[\begin{matrix} S_N & *\\ \mathbb O & R \end{matrix}\right], \end{equation*}
 {\cite[Theorem  XI.3.7]{conway}} implies that $\operatorname{ind}U_{\mathbb T, N+1}|_{\mathcal M}=\operatorname{ind}S_N+\operatorname{ind}R=-N-1$. 
Since $\mu_{U_{\mathbb T, N+1}|_{\mathcal M}}\leq N+1$ (where $\mu_T$ for an operator $T$ is defined in \eqref{mu}), we conclude that $U_{\mathbb T, N+1}|_{\mathcal M}\cong S_{N+1}$.
\end{proof}

Recall that the multiplicity  $\mu_T$ for an operator $T$ is defined in \eqref{mu}. 

\begin{theorem} \label {thmmainiso}Suppose than $N\in\mathbb N$, $V_+\in\mathcal L(\mathcal K_+)$ is an a.c. isometry, 
$\mu_{V_+}\leq N$, 
$X_+\in\mathcal L(\mathcal K_+, H^2_N)$, and $S_N^*X_+V_+=X_+$. Suppose that there exist
 $\{f_j\}_{j=1}^N\subset\mathcal K_+$
such that $X_+f_j=e_j$ ($j=1,\ldots, N$). Then \begin{equation*}V_+|_{\vee_{j=1}^N\vee_{n=0}^\infty V_+^n f_j}\cong S_N.\end{equation*}
\end{theorem}

\begin{proof} The theorem will be proved using induction. 
Let $N=1$. Since there exists $f_1\in\mathcal K_+$ such that $Xf_1=e_1=\mathbf{1}$, we have 
 $\mathcal K_+\neq\{0\}$, and $V_+\cong S$ or $V_+\cong U_\sigma$ for some $\sigma\subset\mathbb T$. If $V_+\cong S$, 
the conclusion of the theorem
is fulfilled for every $0\not\equiv f_1\in\mathcal K_+$. If $V_+\cong U_\sigma$, Lemma \ref{lem11} is applied.   
Thus, if $N=1$, then the theorem is proved.

If $N\geq 1$, assume that the theorem is proved for all $1\leq k\leq N$. 
We will to prove the theorem for $N+1$. Let $X$ and $V$ be from Lemma \ref{lem12} applied to $V_+$ and $X_+$. 
Then $V$ is unitary, and  $S_{N+1}^*XV=X$. Set
\begin{equation} \mathcal M_k=\vee_{j=1}^k\vee_{n\in\mathbb Z}V^n f_j 
\ \ \  \text{ and } X_k=P_{H_k^2\oplus\{0\}}X|_{\mathcal M_k} \ \ \ (k=1,\ldots, N+1). 
\end{equation} 
Then $S_k^*X_kV|_{\mathcal M_k} =X_k$ and $X_kf_j=e_j$ for $j=1, \ldots, k$ ($k=1,\ldots, N+1$). Thus,  
$X_k$ and $V|_{\mathcal M_k}$ satisfy the assumption of the theorem. By the inductive hypothesis, 
$V|_{\vee_{j=1}^k \vee_{n=0}^\infty V^n f_j}\cong S_k$ for $k=1, \ldots, N$. Consequently,
\begin{equation}\label{mmconguu} V|_{\mathcal M_k}\cong U_{\mathbb T,k}\ \ \  (k=1, \ldots, N). \end{equation}

Taking into account relations \eqref{mmconguu} and the estimate $\mu_V\leq N+1$, and using appropriate unitary equivalence, we may assume that $V=U_{\mathbb T,N}\oplus U_\sigma$ 
for some $\sigma\subset \mathbb T$, and $\mathcal M_k=L^2_k\oplus\{0\}\subset L^2_N$ ($k=1, \ldots, N$). 

Write $S_{N+1}^*$ as $(N+1)\times (N+1)$  diagonal   matrix, whose elements on the main diagonal are 
$S^*$. Write  $V$  as $(N+1)\times (N+1)$  diagonal   matrix, whose $N$ elements on the main diagonal are  $U_{\mathbb T}$
and the ending element is $U_\sigma$. Write $X$ as a $(N+1)\times (N+1)$ matrix: 
$X=[X_{jk}]_{j,k=1,\ldots, N+1}$.  
Then  $S^*X_{jk}U_{\mathbb T}=X_{jk}$  and $S^*X_{j, N+1}U_{\sigma}=X_{j, N+1}$ for all $j=1,\ldots, N+1$, $k=1,\ldots,N$. 
Therefore, there exist 
$\psi_{jk}\in L^\infty$ such that $X_{jk}f=P_+\psi_{jk} f$ for every $f\in L^2$, and $X_{j,N+1}f=P_+\psi_{j,N+1} f$ for every $f\in L^2(\sigma,m)$ and for all $j=1,\ldots, N+1$, $k=1,\ldots,N$.

Set $\Psi=[\psi_{jk}]_{j,k=1,\ldots, N+1}$. 
For $k=1,\ldots, N$ write $f_k\in L^2_k\oplus\{0\}$ as a column whose first $k$ elements are functions from $L^2$ and other are zeros functions. 
 Write $f_{N+1}$ as a column whose first $N$ elements are functions from $L^2$ and $(N+1)$th element is a 
function from $L^2(\sigma,m)$. Set $F=[f_1,\ldots,f_{N+1}]$.
 Then $F$ is a upper-triangular $(N+1)\times (N+1)$ matrix, 
whose elements are functions from $L^2$ and $L^2(\sigma,m)$. Denote the elements from the main diagonal of $F$ by $f_{0k}$ ($k=1,\ldots, N+1$).
Then $f_{0k}\in L^2$ for $k=1,\ldots, N$,  $f_{0,N+1}\in L^2(\sigma,m)$, and $\det F=\prod_{k=1}^{N+1}f_{0k}$.

Since $Xf_j=e_j$ ($j=1,\ldots, N+1$), we have $P_+\Psi F=I_{(N+1)\times (N+1)}$. Therefore,
 there exists $(N+1)\times (N+1)$ matrix $G$, whose elements are functions from $H^2$, such that 
$\Psi F=I_{(N+1)\times (N+1)}+\overline\chi\overline G$.
Set $h=\det(I_{(N+1)\times (N+1)}+\chi G)$. Then $h\in H^{\frac{2}{N+1}}$, and $h(0)=1$. Therefore,  
$\int_{\mathbb T}\log|h|\mathrm{d}m>-\infty$. 
Set $\psi=\det \Psi$. Then $\psi\in L^\infty$. We have 
\begin{equation*}\overline h=\det(\Psi F)=\det\Psi\det F=\psi\prod_{k=1}^{N+1}f_{0k}.
\end{equation*}
Therefore, 
\begin{equation*}
\int_{\mathbb T}\log|\psi|\mathrm{d}m+\sum_{k=1}^{N+1}\int_{\mathbb T}\log|f_{0k}|\mathrm{d}m=
\int_{\mathbb T}\log|h|\mathrm{d}m>-\infty.
\end{equation*}
We obtain that $\int_{\mathbb T}\log|f_{0k}|\mathrm{d}m>-\infty$ for all $k=1,\ldots, N+1$. In particular, $\sigma=\mathbb T$ 
and $V=U_{\mathbb T,N+1}$. 

By the inductive hypothesis, $V|_{\vee_{j=1}^N \vee_{n=0}^\infty V^n f_j}\cong S_N$. We may assume that 
\begin{equation*}\vee_{j=1}^N \vee_{n=0}^\infty V^n f_j=H^2_N\oplus\{0\}\oplus\{0\}\subset 
H^2_N\oplus (H^2_-)_N\oplus L^2=L^2_{N+1}. \end{equation*}  
Note that  $f_{0,N+1}=P_{\{0\}\oplus\{0\}\oplus L^2}f_{N+1}$. 
Set $\overline\chi\overline h_0=P_{\{0\}\oplus (H^2_-)_N\oplus\{0\}}f_{N+1}$, where $h_0\in H^2_N$.  
Then \begin{equation*}\vee_{j=1}^{N+1} \vee_{n=0}^\infty V^n f_j 
=H^2_N\vee\vee_{n=0}^\infty V^n(\overline\chi\overline h_0\oplus f_{0,N+1}).
\end{equation*}
By Lemma \ref{lem13}, \begin{equation*}V|_{\vee_{j=1}^{N+1} \vee_{n=0}^\infty V^n f_j}\cong S_{N+1}.\end{equation*}
Since $V$ is a unitary extension of $V_+$, the theorem is proved. 
\end{proof}

Let $\psi\in L^\infty$. The \emph{Hankel operator} $H_\psi\in\mathcal L(H^2,H^2_-)$ with the symbol $\psi$ acts by the formula 
$H_\psi h=P_-\psi h$ ($h\in H^2$). By {\cite[formula (1.1.9)]{peller}}, $\|H_\psi\|=\operatorname{dist}(\psi, H^\infty)$.
If $\theta_k$ ($k=1,2)$ are inner functions, then 
\begin{equation}\label{hankel}\|P_{\mathcal K_{\theta_1}}|_{\theta_2 H^2}\|=\|H_{\overline{\theta_1}\theta_2}\|=
\operatorname{dist}(\theta_1,\theta_2 H^\infty)\leq\|\theta_1-\theta_2\|_\infty.
\end{equation}

For an inner function $\theta\in H^\infty$ and $0\neq a\in\mathbb D$ set 
\begin{equation}\label{thetaa} \theta_a=\frac{\theta-a}{1-\overline a\theta}.
\end{equation}
Then $\theta_a$ is an inner function, $\theta$ and $\theta_a$ are relatively prime, and 
\begin{equation}\label{estthetaa} \|\theta-\theta_a\|_\infty\leq\frac{2|a|}{1-|a|}. 
\end{equation}

\begin{lemma}\label{lem27} Suppose that 
$N\in\mathbb N$, $N\geq 2$, 
 $\delta_0>0$, $\theta\in H^\infty$ is an inner function, $\mathcal H$ is a Hilbert space, and 
$Z\in\mathcal L(H^2_N, \mathcal H)$ is such that 
\begin{equation*}\|Z(\theta h\oplus\{0\})\|\geq\delta_0\|h\| \ \ \text{  for every } \ h\in H^2.\end{equation*} 
Then for every $0<\delta<\delta_0$ there exist $\{\mathcal N_j\}_{j=1}^N\subset\operatorname{Lat}S_N$ such that 
$S_N|_{\mathcal N_j}\cong S$, 
$\|Zh\|\geq\delta\|h\|$ for every $h\in \mathcal N_j$ and $j=1,\ldots, N$, and $\vee_{j=1}^N\mathcal N_j=H^2_N$.
\end{lemma}

\begin{proof} Let $0\neq a\in\mathbb D$, and let $0<\varepsilon<1$. Define $N\times N$ matrix $\Theta$ as follows:  
\begin{equation*} \Theta=\left[\begin{matrix}
(1-\varepsilon^2)^{\frac{1}{2}}\theta_a & (1-\varepsilon^2)^{\frac{1}{2}}\theta & 
(1-\varepsilon^2)^{\frac{1}{2}}\theta &\ldots & (1-\varepsilon^2)^{\frac{1}{2}}\theta \\
\varepsilon & \varepsilon & 0 & \ldots & 0 \\ 
0 & 0 &  \varepsilon & \ldots & 0 \\  \ldots &  \ldots &  \ldots &   \ldots &  \ldots \\ 
0 & 0 & 0 &  \ldots &  \varepsilon \end{matrix}\right].
\end{equation*}
Then 
\begin{align*} \det \Theta &=\det  \left[\begin{matrix}(1-\varepsilon^2)^{\frac{1}{2}}\theta_a & (1-\varepsilon^2)^{\frac{1}{2}}\theta \\
\varepsilon & \varepsilon\end{matrix}\right]
\det\left[\begin{matrix}\varepsilon & \ldots & 0 \\   \ldots &   \ldots &  \ldots \\ 
 0 &  \ldots &  \varepsilon\end{matrix}\right] \\ &
= (1-\varepsilon^2)^{\frac{1}{2}}\varepsilon^{N-1}(\theta_a-\theta)=-(1-\varepsilon^2)^{\frac{1}{2}}\varepsilon^{N-1}
\frac{a(1-\frac{\overline a}{a}\theta^2)}{1-\overline a \theta}.
\end{align*} 
Therefore, $\det\Theta$  is an outer function. By {\cite[Prop. V.6.1 and Theorem V.6.2]{nfbk}}, $\Theta$ is an outer function. 
The columns $\Theta_j$ ($j=1,\ldots, N$) of the matrix $\Theta$ are inner functions from $H^\infty(\mathbb C, \mathbb C^N)$. 
Set $\mathcal N_j=\Theta_j H^2$ ($j=1,\ldots, N$). Then $\mathcal N_j\in\operatorname{Lat}S_N$, 
$S_N|_{\mathcal N_j}\cong S$,  ($j=1,\ldots, N$), 
and $\vee_{j=1}^N\mathcal N_j=H^2_N$.

Let $h\in H^2$. For $2\leq j \leq N$ we have 
\begin{align*} \|Z\Theta_j h\|&\geq\|Z(1-\varepsilon^2)^{\frac{1}{2}}(\theta h\oplus\{0\})\|-\|Z(0\oplus\ldots\oplus\varepsilon h\oplus\ldots \oplus 0)\|\\&
\geq(1-\varepsilon^2)^{\frac{1}{2}}\delta_0\|h\|-\|Z\|\varepsilon \|h\|\\&=
\bigl((1-\varepsilon^2)^{\frac{1}{2}}\delta_0-\|Z\|\varepsilon\bigr) \|h\|=
\bigl((1-\varepsilon^2)^{\frac{1}{2}}\delta_0-\|Z\|\varepsilon\bigr) \|\Theta_jh\|.
\end{align*}
By \eqref{hankel} and \eqref{estthetaa}, $\|P_{\mathcal K_\theta}\theta_a h\|\leq\frac{2|a|}{1-|a|}\|h\|$. Therefore, 
\begin{equation*} \|P_{\theta H^2}\theta_a h\|^2=\|h\|^2-\|P_{\mathcal K_\theta}\theta_a h\|^2\geq \frac{1-2|a|-3|a|^2}{(1-|a|)^2}\|h\|^2.
\end{equation*}
For $j=1$  we have 
\begin{align*} \|Z\Theta_1 h\|&\geq\|Z(1-\varepsilon^2)^{\frac{1}{2}}(P_{\theta H^2}\theta_a h\oplus\{0\})\|\\&\quad
-\|Z(1-\varepsilon^2)^{\frac{1}{2}}(P_{\mathcal K_\theta}\theta_a h\oplus\{0\})\|
-\|Z(0\oplus\varepsilon h\oplus\ldots \oplus 0)\|\\&
\geq(1-\varepsilon^2)^{\frac{1}{2}}\delta_0\frac{(1-2|a|-3|a|^2)^{\frac{1}{2}}}{1-|a|}\|h\|\\&\quad -
\|Z\|(1-\varepsilon^2)^{\frac{1}{2}}\frac{2|a|}{1-|a|}\|h\|-
\|Z\|\varepsilon \|h\|\\&=
\Bigl(\frac{(1-\varepsilon^2)^{\frac{1}{2}}}{1-|a|}(\delta_0(1-2|a|-3|a|^2)^{\frac{1}{2}} -2|a|\|Z\|)-\|Z\|\varepsilon\Bigr) \|h\|\\&=
\Bigl(\frac{(1-\varepsilon^2)^{\frac{1}{2}}}{1-|a|}(\delta_0(1-2|a|-3|a|^2)^{\frac{1}{2}} -2|a|\|Z\|)-\|Z\|\varepsilon\Bigr) \|\Theta_1h\|.
\end{align*}
When $0<\delta<\delta_0$ is given, the conclusion of the lemma is fulfilled for sufficiently small $|a|$ and $\varepsilon$.
\end{proof}

\begin{theoremcite}\label{theorembb}\cite{hjelle} Let $u\in L^\infty$, and let $|u|=1$ $m$-a.e. 
on $\mathbb T$. Then for every $\varepsilon>0$ there exist  $\alpha$, $\beta$, $\varphi\in H^\infty$ such that $\alpha$ and $\beta$ are inner,  $\frac{1}{\varphi}\in H^\infty$, $\|\varphi\|_\infty\leq 1+\varepsilon$, 
$\|\frac{1}{\varphi}\|_\infty\leq 1+\varepsilon$, 
and 
\begin{equation*}
u=\frac{\overline\varphi}{\varphi}\alpha\overline\beta. 
\end{equation*}
\end{theoremcite}

\subsection{Resuts}

In this subsection main results of the paper are proved.

\begin{lemma}\label{lemmain0} Suppose that  $T$ is an expansive operator,  $N=\dim\ker T^*<\infty$, 
and $S_N\prec T$.  
 Set $\mathcal H_1= \vee_{n=0}^\infty T'^n\ker T^*$. 
Then $T'$ is an a.c.  contraction of class $C_{1\cdot}$, and $(T'|_{\mathcal H_1})_+^{(a)}\cong S_N$.
\end{lemma}  
\begin{proof} By Lemma \ref{lemnotuu}(ii), $T'$ is an a.c. contraction. 
Denote by $Y$  a quasiaffinity such that $YT^*=S_N^*Y$. Then $Y\ker T^*=\ker S_N^*$ and 
$Y=S_N^*YT'$. By Lemma \ref{lemyy}(ii), $T'$ is a contraction of class $C_{1\cdot}$. 

 Set  $V_+=(T')^{(a)}_+$. By {\cite[Sec. IX.1]{nfbk}}, $V_+$ is an a.c. isometry. Let $X_+$ be from Lemma \ref{lemyy}.
Then $Y=X_+X_{+,T'}$ and  $X_+=S_N^*X_+V_+$. Set $\mathcal F=X_{+,T'}\ker T^*$. Then  $\ker S_N^*=Y\ker T^*=X_+\mathcal F$. By Theorem \ref{thmmainiso}, 
\begin{equation*} V_+|_{\vee_{n=0}^\infty V_+^n\mathcal F}\cong S_N.\end{equation*}
By \cite{ker89},  $(T'|_{\mathcal H_1})_+^{(a)}=V_+|_{\operatorname{clos}X_{+,T'}\mathcal H_1}$. Furthermore,  
  \begin{equation*}  
\operatorname{clos}X_{+,T'}\mathcal H_1=\operatorname{clos}X_{+,T'}(\vee_{n=0}^\infty T'^n\ker T^*)=
\vee_{n=0}^\infty V_+^n\mathcal F.
\end{equation*}
Thus, $(T'|_{\mathcal H_1})_+^{(a)}\cong S_N$.  
\end{proof}

\begin{lemma}\label{lemc0} Suppose that  $T$ is an expansive operator,   $N=\dim\ker T^*<\infty$, and $S_N\prec T$.
  Then   $T'$ is a contraction of class $C_{10}$, and $(T')_+^{(a)}\cong S_N$.
\end{lemma}  

\begin{proof} Denote by $\mathcal H$ the space on which $T$ acts.  Set $\mathcal H_1= \vee_{n=0}^\infty T'^n\ker T^*$,  $T_1=T'|_{\mathcal H_1}$, 
$\mathcal H_0=\mathcal H\ominus\mathcal H_1$. Then 
\begin{equation*} T'=\left[\begin{matrix}T_1 & T_2 \\ \mathbb O & T_0\end{matrix}\right]
\end{equation*}
with respect to the decomposition $\mathcal H=\mathcal H_1\oplus\mathcal H_0$. Note that $\ker T'^*=\ker T_1^*$. 
We will to prove that $T_0$ is a $C_0$-contraction. 

Let $Z_0\in\mathcal  I(S, T_0)$.  
By {\cite[Lemma 1]{takshiftindex}}, there exists $Z_2\in\mathcal  L(H^2,\mathcal H_1)$ 
such that  \begin{equation*} \left[\begin{matrix}I_{\mathcal H_1} & Z_2 \\ \mathbb O & Z_0\end{matrix}\right]
\left[\begin{matrix}T_1 & \mathbb O \\ \mathbb O & S\end{matrix}\right] = 
\left[\begin{matrix}T_1 & T_2 \\ \mathbb O & T_0\end{matrix}\right]
\left[\begin{matrix}I_{\mathcal H_1} & Z_2 \\ \mathbb O & Z_0\end{matrix}\right].
\end{equation*}
Let $Z_1\in\mathcal  I(S_N, T_1)$ be from Lemma \ref{lemmodel} applied to $T_1$. Then $Z_1\ker S_N^*=\ker T_1^*=\ker T^*$. 
 Since $\operatorname{clos}Z_1H^2_N=\mathcal H_1$, 
 Lemma \ref{lemfred} implies that  
\begin{equation}\label{kerzz1} \ker Z_1=\{0\}.
\end{equation}  
It is easy to see that  \begin{equation*} \left[\begin{matrix}Z_1 & Z_2 \\ \mathbb O & Z_0\end{matrix}\right]
\left[\begin{matrix}S_N & \mathbb O \\ \mathbb O & S\end{matrix}\right] = 
\left[\begin{matrix}T_1 & T_2 \\ \mathbb O & T_0\end{matrix}\right]
\left[\begin{matrix}Z_1 & Z_2 \\ \mathbb O & Z_0\end{matrix}\right].\end{equation*}
Let $Y$ be a quasiaffinity such that $Y^*S_N=TY^*$.
 Since $\dim\ker T^*=N$, we have $Y\ker T^*=\ker S_N^*$. 

Set \begin{equation*} Z=\left[\begin{matrix} Z_1 & Z_2 \\ \mathbb O & Z_0\end{matrix}\right] 
\end{equation*}
and $Z_+=YZ$. 
We have $S_N^*Z_+S_{N+1}=Z_+$. 
Since $Z_1\ker S_N^*=\ker T^*$, we have $Z_+(\ker S_N^*\oplus\{0\})=\ker S_N^*$. 
Therefore, 
\begin{align*}S_NZ_+=S_NS^*_NZ_+S_{N+1}=(I_{H^2_N}-P_{\ker S_N^*})Z_+S_{N+1}\\=
(I_{H^2_N}-P_{Z_+(\ker S_N^*\oplus\{0\})})Z_+S_{N+1}=Z_+\Bigl(S_{N+1}+\sum_{k=1}^N e_k\otimes f_k\Bigr)
\end{align*}
for some $\{f_k\}_{k=1}^N\subset H^2_N$. 
Set $A=S_{N+1}+\sum_{k=1}^N e_k\otimes f_k$. 
By Lemma \ref{lemma1}, $S_{N+1}\buildrel i\over\prec A$. If $\ker Z_+=\{0\}$, then 
$S_{N+1}\buildrel i\over\prec A\buildrel i\over\prec S_N$, a contradiction. Thus,  $\ker Z_+\neq\{0\}$. Consequently, 
$\ker Z\neq\{0\}$.  
 From this relation, \eqref{kerzz1} and the definition of $Z$ we conclude that $\ker Z_0\neq\{0\}$. 
By {\cite[Introduction]{takshiftindex}},  $T_0$ is a $C_0$-contraction. 

By Lemmas \ref{lemasymp1} and \ref{lemmain0},  $(T')_+^{(a)}\cong S_N$. Since $T'$ is of class $C_{1\cdot}$, we have $T'\prec S_N$.  Therefore, 
$T'$ is of class $C_{10}$.  
\end{proof}

\begin{theorem}\label{thmmain0} Suppose that  $T$ is an expansive operator,   $N=\dim\ker T^*<\infty$, and
 $S_N\buildrel d\over\prec T$.
 Then  $I-T^*T\in\frak S_1$, $T\sim S_N$, and $T'\sim S_N$.  
\end{theorem}  

\begin{proof} 
By Lemma \ref{lemfred}, there exists a quasiaffinity $Y$
such that  $YT^*=S_N^*Y$. Consequently, $Y\ker T^*=\ker S_N^*$.  
Furthermore, $Y=S_N^*YT'$.  By Lemma \ref{lemc0}, $(T')_+^{(a)}\cong S_N$. By Lemma \ref{lemyy}, there exists 
$X_+\in\mathcal L(H^2_N)$ such that  $Y=X_+X_{+,T'}$ and $S_N^*X_+S_N=X_+$. By Lemma \ref{lemcc},  there exists 
an $N\times N$ matrix $\Psi$ whose elements are functions from $L^\infty$ such that $X_+=T_\Psi$. 

Set $\mathcal H_1= \vee_{n=0}^\infty T'^n\ker T^*$ and  $T_1=T'|_{\mathcal H_1}$. Note that $\ker T^*=\ker T_1^*$.
Let $Z\in \mathcal I(S_N, T_1)$ be from Lemma \ref{lemmodel} applied to $T_1$.
Then $Z\ker S_N^*=\ker T^*$. By Lemma \ref{lemfred}, $Z$ is a quasiaffinity. 

Denote by $\mathcal H$ the space on which $T$ acts, and by $J$ the natural imbedding of 
 $\mathcal H_1$ into $\mathcal H$. 
Since $X_{+,T'}JZS_N=S_NX_{+,T'}JZ$, there exists an $N\times N$ matrix $\Phi$ whose elements are functions from $H^\infty$ such that 
$X_{+,T'}JZ=T_\Phi$. Let $\Phi=\Theta_0\Phi_0=\Phi_1\Theta_1$ be the canonical and $*$-canonical factorizations of 
operator-valued function 
$\Phi$ {\cite[Sec. V.4.3]{nfbk}}. Namely, $\Theta_0$ is inner, $\Phi_0$ is outer, $\Phi_1$ is $*$-outer, and $\Theta_1$ is $*$-inner. We have 
\begin{equation}\label{closhh1}\operatorname{clos}X_{+,T'}\mathcal H_1=\operatorname{clos}X_{+,T'}JZ H^2_N=\operatorname{clos}\Phi H^2_N=\Theta_0 H^2_M
\end{equation}
for some $M\leq N$. By Lemma \ref{lemmain0}, $(T'|_{\mathcal H_1})_+^{(a)}\cong S_N$. Consequently, $M=N$. By {\cite[Secs. V.6.1, V.6.2]{nfbk}}, 
$\Theta_0$ and $\Theta_1$ are inner from both sides, $\Phi_0$ and $\Phi_1$ are outer from both sides, and $\varphi:=\det \Phi_0=\det\Phi_1$ 
is outer. Clearly, $\varphi\in H^\infty$. 
Furthermore, $\Phi$ is outer if and only if both $\Theta_0$ and $\Theta_1$ are unitary constant functions. 
Assume that $\Theta_1$ is a non-constant inner function. Set 
\begin{equation*}\mathcal K_{\Theta_1}=H^2_N\ominus\Theta_1 H^2_N=\Theta_1 (H^2_-)_N\cap H^2_N.
\end{equation*}

 The equalities $YJZ\ker S_N^*=\ker S_N^*$ and $YJZ=T_{\Psi\Phi}$ imply that 
$\Psi\Phi=\overline G$, where $G$ is an $N\times N$ matrix whose elements are functions from $H^\infty$. 
Consequently,
\begin{equation} \label{kerxxplus} \Phi_1\mathcal K_{\Theta_1}\subset\Phi(H^2_-)_N\cap H^2_N\subset\ker X_+.
\end{equation}
Furthermore, 
 \begin{equation} \label{phi1} X_{+,T'}\mathcal H\cap\Phi_1 H^2_N\subset\Phi H^2_N.
\end{equation} 
Indeed, let $x\in\mathcal H$ be such that $X_{+,T'}x\in\Phi_1 H^2_N$. Then there exist $h\in H^2_N$ and 
 $f\in\mathcal K_{\Theta_1}$ such that  $X_{+,T'}x=\Phi_1\Theta_1 h+\Phi_1 f= \Phi h+\Phi_1 f$.  
Since $\Phi h=X_{+,T'}JZh$, we have   $\Phi_1 f=X_{+,T'}(x-JZh) $.  By \eqref{kerxxplus}, $Y(x-JZh)=0$. Since $\ker Y=\{0\}$, we have 
$\Phi_1 f\equiv 0$. Since $\Phi_1$ is an outer $N\times N$ matrix-valued function, we conclude that $f\equiv 0$. The inclusion \eqref{phi1} is proved.

Let $\Phi_1^{\mathrm{Ad}}$ be the (algebraic) adjoint of $\Phi_1$. Then $\Phi_1\Phi_1^{\mathrm{Ad}}=\varphi I_{N\times N}$. 
Consequently, 
\begin{equation*} \Phi_1\Phi_1^{\mathrm{Ad}}X_{+,T'}\mathcal H=\varphi(S_N)X_{+,T'}\mathcal H
=X_{+,T'}\varphi(T')\mathcal H\subset 
X_{+,T'}\mathcal H\cap\Phi_1 H^2_N\subset\Phi H^2_N
\end{equation*} 
by \eqref{phi1}. Since $\varphi$ is outer, we have 
\begin{equation*}H^2_N=\operatorname{clos}\varphi(S_N)H^2_N=
\operatorname{clos}\varphi(S_N)X_{+,T'}\mathcal H\subset\operatorname{clos}\Phi H^2_N.\end{equation*}
The latest inclusion and  \eqref{closhh1} imply that $\operatorname{clos} X_{+,T'}\mathcal H_1 =H^2_N$.  Since the mapping 
$\mathcal M\mapsto\operatorname{clos} X_{+,T'}\mathcal M$ ($\mathcal M\in\operatorname{Lat}T'$) 
 is a lattice-isomorphism between $\operatorname{Lat}T'$ and   $\operatorname{Lat}S_N$ \cite{gamal02},
 we conclude that $\mathcal H_1=\mathcal H$. Thus, $T'=T_1$ and the relation $T'\sim S_N$ is proved. 
Furthermore, by Lemma \ref{lemshimorin}(i),  $\mathcal R^\infty(T)=\{0\}$. By Corollary \ref{cormodel}, $T\sim S_N$. 

By \cite{gamal}, 
$I-T'^*T'\in\frak S_1$. By Lemma \ref{lemdefect},  $I-T^*T\in\frak S_1$. 
\end{proof}

\begin{lemma}\label{lemrr1}
Suppose that  $T$
 is  an expansive operator,   $\dim\ker T^*=1$, and $S\buildrel d\over\prec T$.  
Then there exists a quasiaffinity $Z_1\in \mathcal I(S,T)$ such that
  $\|Z_1\|=1$ and for every $0<\delta<1$ there exists an inner function $\vartheta$ 
such that  $\|Z_1\vartheta h\|\geq\delta \|h\|$ for every $h\in H^2$.
\end{lemma}

\begin{proof} We repeat the part of the proof of  Theorem \ref{thmmain0}. 

By Lemma \ref{lemfred}, there exists a quasiaffinity $Y$
such that  $YT^*=S^*Y$. Consequently, there exists  $x_0\in\ker T^*$ such that   $Yx_0= \mathbf{1}$. 
Furthermore, $Y=S^*YT'$.  By Lemma \ref{lemc0}, $(T')_+^{(a)}\cong S$. By Lemma \ref{lemyy}, there exists 
$X_+\in\mathcal L(H^2)$ such that  $Y=X_+X_{+,T'}$ and $S^*X_+S=X_+$. By Lemma \ref{lemcc},  there exists $\psi\in L^\infty$ such that $X_+=T_\psi$. 

Denote by $\mathcal H$ the space on which $T$ acts. By Lemma \ref{lemshimorin}(i) and Theorem \ref{thmmain0}, 
\begin{equation}\label{ttx0} \mathcal H= \vee_{n=0}^\infty T'^nx_0.\end{equation}  
Let $Z\in \mathcal I(S, T')$ be from Lemma \ref{lemmodel} applied to $T'$. 
Multiplying $Z$ by an appropriate constant, we may assume that     $Z\mathbf{1}=x_0$. 
By Lemma \ref{lemfred}, $Z$ is a quasiaffinity. 

Set $\varphi_0=X_{+,T'}x_0$. By \eqref{ttx0},  $\varphi_0\in H^2$ is an outer function. 
Since $X_{+,T'}x_0=X_{+,T'}Z\mathbf{1}$, we conclude that $X_{+,T'}Z=\varphi_0(S)$. Therefore,
 $\varphi_0\in  H^\infty$. 
Furthermore, \begin{equation*} 0=S^*\mathbf{1}=S^*Yx_0=S^*X_+X_{+,T'}x_0=S^*T_\psi \varphi_0.
\end{equation*}
Consequently, $\overline{\psi\varphi_0}\in H^\infty$.  
Therefore, there exists $\eta\in H^\infty$ such that $\psi=\overline\eta\frac{\overline\varphi_0}{\varphi_0}$. 

We prove that $\eta$ is outer. Indeed, assume that $\eta=\theta g$, where $1\not\equiv\theta$ is inner. 
Let $0\not\equiv f\in \mathcal K_\theta$.  Set $y=Zf$. Then
\begin{equation*}Yy=X_+ X_{+,T'}y=X_+ X_{+,T'}Zf= T_\psi  \varphi_0 f=P_+\overline{\theta g}\frac{\overline\varphi_0}{\varphi_0}\varphi_0 f=
P_+\overline{\theta g\varphi_0} f=0,
\end{equation*}
because $f\in\mathcal K_\theta$. This contradicts with the equality $\ker Y=\{0\}$.  Thus, $\eta$ is outer.

Set $Z_1=X_{+,T'}^*T_{\frac{\varphi_0}{\overline\varphi_0}}$. Then $\|Z_1\|\leq 1$ and $Z_1\eta(S)=Y^*$. 
Therefore, $Z_1S=TZ_1$ and $\operatorname{clos}Z_1H^2=\mathcal H$. By Lemma \ref{lemfred}, $\ker Z_1=\{0\}$.

Let $0<\delta<1$. Take  $0<\delta_1<1$ and   $\varepsilon_1>0$ such that 
$\frac{\delta_1}{(1+\varepsilon_1)^2}\geq\delta$.
By {\cite[Ch. IX.3]{nfbk}} or \cite{ker07}, there exists $\mathcal M\in\operatorname{Lat}T'$ such that 
$\|X_{+,T'}x\|\geq\delta_1\|x\|$ for every $x\in\mathcal M$. 
Therefore, there exists an inner function $\vartheta_0$ such that $X_{+,T'}{\mathcal M}=\vartheta_0 H^2$. 
Consider $X_{+,T'}|_{\mathcal M}$ as a transformation from $\mathcal L(\mathcal M,\vartheta_0 H^2)$.  
Then  $X_{+,T'}|_{\mathcal M}$ is invertible, and  
\begin{equation*} \|((X_{+,T'}|_{\mathcal M})^*)^{-1}\|=\| (X_{+,T'}|_{\mathcal M})^{-1}\|\leq 1/\delta_1.\end{equation*}
Consequently,  
\begin{equation*}
 \|X_{+,T'}^*h\|\geq \|P_{\mathcal M}X_{+,T'}^*h\|= \|P_{\mathcal M}X_{+,T'}^*P_{\vartheta_0 H^2}h\|\geq\delta_1\|P_{\vartheta_0 H^2}h\|
\end{equation*}
for every $h\in H^2$.

Let   $\alpha$, $\beta$, $\varphi$ be from Theorem \ref{theorembb} applied to 
$\frac{\varphi_0}{\overline\varphi_0}$ and $\varepsilon_1$. 
Let $h\in H^2$. Then
\begin{align*} \|Z_1\vartheta_0\beta h \|=\|X_{+,T'}^*T_{\frac{\varphi_0}{\overline\varphi_0}}\vartheta_0\beta h\|\geq
\delta_1\|P_{\vartheta_0 H^2}T_{\frac{\overline\varphi}{\varphi}\alpha\overline\beta}\vartheta_0\beta h\|=
\delta_1\|T_{\overline\vartheta_0}T_{\frac{\overline\varphi}{\varphi}\alpha\overline\beta}\vartheta_0\beta h\|\\=
\delta_1\|P_+\frac{\overline\varphi}{\varphi}\alpha h\|\geq \delta_1\frac{1}{\|\varphi\|_\infty\|\frac{1}{\varphi}\|_\infty}\|\alpha h\|\geq\frac{\delta_1}{(1+\varepsilon_1)^2}\|\vartheta_0\beta h\|\geq\delta\|\vartheta_0\beta h\|.
\end{align*}
Setting $\vartheta=\vartheta_0\beta$, we conclude that $Z_1$  satisfies the conclusion of the lemma. 
\end{proof}

\begin{theorem}\label{thm1} Suppose that  $T\in\mathcal L(\mathcal H)$ is expansive, $\dim\ker T^*=1$,  and 
$S\buildrel d\over\prec T$. 
 Then for every 
$\varepsilon>0$ there exist $\mathcal M_1$, $\mathcal M_2\in\operatorname{Lat}T$ 
and invertible transformations  $Y_1$, $Y_2$ such that 
\begin{equation*}Y_kS=T|_{\mathcal M_k}Y_k, \ \ \|Y_k\|\|Y_k^{-1}\|\leq 1+\varepsilon \ (k=1,2), 
\ \text{ and }  \ \mathcal M_1\vee\mathcal M_2=\mathcal H. 
\end{equation*}
\end{theorem}

\begin{proof} 
Let $Z_1$ be  a quasiaffinity from Lemma \ref{lemrr1}.  Let $\vartheta_1$ be an inner function from 
Lemma \ref{lemrr1} applied with $\delta_1>1/(1+\varepsilon)$. 
Let $\varepsilon_1>0$ be such that 
$\delta_1(1-\varepsilon_1^2)^{\frac{1}{2}}-\varepsilon_1\geq 1/(1+\varepsilon)$. Take 
an inner function $\vartheta_2$ such that $\vartheta_1$ and $\vartheta_2$ are relatively prime and 
$\|\vartheta_1-\vartheta_2\|_\infty \leq\varepsilon_1$ (for example, use \eqref{thetaa}).
By \eqref{hankel}, $\|P_{\mathcal K_{\vartheta_1}}|_{\vartheta_2 H^2}\|\leq\varepsilon_1$.

 Let $h\in H^2$. Then 
$\|P_{\vartheta_1 H^2}\vartheta_2 h\|^2=\|h\|^2-\|P_{\mathcal K_{\vartheta_1}}\vartheta_2 h\|^2\geq(1-\varepsilon_1^2)\|h\|^2$. 
Therefore,
\begin{align*} \|Z_1\vartheta_2 h\|&\geq\|Z_1P_{\vartheta_1 H^2}\vartheta_2 h\|
-\|Z_1P_{\mathcal K_{\vartheta_1}}\vartheta_2 h\|
\geq\delta_1\|P_{\vartheta_1 H^2}\vartheta_2 h\|-\varepsilon_1\|h\|\\&
\geq\delta_1(1-\varepsilon_1^2)^{\frac{1}{2}}\|h\|-\varepsilon_1\|h\|=
\bigl(\delta_1(1-\varepsilon_1^2)^{\frac{1}{2}}-\varepsilon_1)\|h\|\geq \|h\|/(1+\varepsilon).
\end{align*}
Since  $\|Z_1\vartheta_1 h\|\geq\delta_1\|h\|$, we obtain that 
\begin{equation*}\|Z_1\vartheta_k h\|\geq\|h\|/(1+\varepsilon) \ \ \text{ for every }h\in H^2\ \text{ and } k=1,2.
\end{equation*}
Set $\mathcal M_k=Z_1\vartheta_kH^2$ ($k=1,2$). Consider  $Y_k=Z_1|_{\vartheta_k H^2}$ as the transformations from 
 $\mathcal L(\vartheta_k H^2,\mathcal M_k)$. Then $T|_{\mathcal M_k}=Y_k S|_{\vartheta_k H^2}Y_k^{-1}$ and 
 $\|Y_k\|\|Y_k^{-1}\|\leq1+\varepsilon$. Clearly, $S|_{\vartheta_k H^2}\cong S$. 
Thus, $\mathcal M_k$ and $Y_k$ (up to appropriate unitary equivalence) ($k=1,2$) satisfy the conclusion of the theorem.
\end{proof}

\begin{theorem} \label{thmnn} Suppose that $N\in\mathbb N$, $N\geq 2$,  $T\in\mathcal L(\mathcal H)$ is expansive, 
$\ker T^*\neq\{0\}$, 
and $S_N\buildrel d\over\prec T$. 
Then for every 
$\varepsilon>0$ there exist $\{\mathcal M_j\}_{j=1}^N\subset\operatorname{Lat}T$ 
and invertible transformations  $Y_j$ such that
\begin{equation*}Y_jS=T|_{\mathcal M_j}Y_j, \ \ \|Y_j\|\|Y_j^{-1}\|\leq 1+\varepsilon \ (j=1,\ldots,N), 
\ \text{ and }  \ \vee_{j=1}^N\mathcal M_j=\mathcal H. 
\end{equation*}

\end{theorem}

\begin{proof} Let $Y$ realize the relation $S_N\buildrel d\over\prec T$. For $1\leq k\leq N$ 
set $\mathcal H_k=\operatorname{clos}Y(\{0\}\oplus \ldots \oplus H^2\oplus\ldots\oplus\{0\})$ (where a unique 
 nonzero summand $H^2$ is on $k$th place), and 
$T_k=T|_{\mathcal H_k}$. If $\ker T_k^*=\{0\}$,   then 
$\mathcal H_k\subset \mathcal R^\infty(T)$. If $\ker T_k^*=\{0\}$ for all $k=1,\ldots,N$, then 
$\mathcal H=\vee_{k=1}^N\mathcal H_k\subset\mathcal R^\infty(T)$, a contradiction with the assumption 
$\ker T^*\neq\{0\}$. 
 Consequently, there exists $1\leq k\leq N$ such that $\ker T_k^*\neq\{0\}$. Without loss of generality we may assume that $k=1$.
 Then $T_1$ satisfies the assumptions of Lemma \ref{lemrr1}, because
the relation $S\buildrel d\over\prec T_1$ implies $\dim\ker T_1^*\leq 1$. 

Let $Z_1$  be  a quasiaffinity from Lemma \ref{lemrr1} applied to $T_1$. 
Take $\varepsilon_1>0$ such that $(1+\varepsilon_1)^2\leq 1+\varepsilon$. 
 Define $Z\in\mathcal L(H^2_N, \mathcal H)$ as follows: 
\begin{equation*} Z|_{H^2\oplus\{0\}}=Z_1, \ \ \ Z|_{\{0\}\oplus H^2_{N-1}}=\frac{\varepsilon_1}{\|Y\|}Y|_{\{0\}\oplus H^2_{N-1}}.
\end{equation*}
Then $ZS_N=TZ$, $\|Z\|\leq1+\varepsilon_1$, $\operatorname{clos}ZH^2_N=\mathcal H$,  and $Z$ satisfies the assumption of Lemma \ref{lem27} 
for every $0<\delta_0<1$ with some inner function $\theta$ (which depends on $\delta_0$). Let  $\{\mathcal N_j\}_{j=1}^N\subset\operatorname{Lat}S_N$ be from Lemma \ref{lem27}
applied with  $\delta\geq 1/(1+\varepsilon_1)$. Set $\mathcal M_j=Z\mathcal N_j$ ($j=1,\ldots, N$). 
 Consider  $Y_j=Z|_{\mathcal N_j}$ as the transformations from 
 $\mathcal L(\mathcal N_j,\mathcal M_j)$. Then $T|_{\mathcal M_j}=Y_j S_N|_{\mathcal N_j}Y_j^{-1}$
 and 
 $\|Y_j\|\|Y_j^{-1}\|\leq (1+\varepsilon_1)^2\leq 1+\varepsilon$. By Lemma \ref{lem27},  $S_N|_{\mathcal N_j}\cong S$.  
Thus, $\mathcal M_j$ and $Y_j$ (up to appropriate unitary equvalence) ($j=1,\ldots, N$) satisfy the conclusion of the theorem.
\end{proof}

\begin{theorem}\label{thmmain1} Suppose that  $T$ is an expansive operator, 
  $N=\dim\ker T^*<\infty$, $I-T^*T\in\frak S_1$, and 
 $\mathcal R^\infty(T)=\{0\}$. Then $T'\sim S_N$, $S_N\buildrel i\over\prec T\prec S_N$ and for every $\mathcal M\in\operatorname{Lat}T$ 
$\dim( \mathcal M\ominus T\mathcal M)\leq N$.
\end{theorem}  

\begin{proof} By Corollary \ref{cormodel}, $T\prec S_N$. Denote by $\mathcal H$ the space in which $T$ acts. By Lemma  \ref{lemshimorin}(i),  
\begin{equation*} \mathcal H=\vee_{n=0}^\infty T'^n\ker T'^*.
\end{equation*}
Therefore,  $\mu_{T'}=N$ (where $\mu_T$ for an operator $T$ is defined in \eqref{mu}), and $S_N\prec T'$ by Lemmas \ref{lemmodel} and \ref{lemfred}. In particular, $T'$ is an a.c. contraction. 

Furthermore, $T'$ is left-invertible, $\dim\ker T'^*=N<\infty$, and $I-T'^*T'\in\frak S_1$ by Lemma \ref{lemdefect}. By \cite{uchiyama84} and \cite{uchiyama83}  or \cite{takahashi84}, $T'$ has the form  
\begin{equation*} T'=\left[\begin{matrix}T_0 & * \\ \mathbb O & T_1\end{matrix}\right],
\end{equation*}
where $T_0$ is a weak contraction (see {\cite[Ch. VIII]{nfbk}} for definition) and $T_1\prec S_N$. 
By {\cite[Lemma 2.1]{gamalshiftindex}}, $(T_1)^{(a)}_+\cong S_N$. 
 By  {\cite[Ch. IX.1]{nfbk}}, 
$T'^{(a)}\cong T_0^{(a)}\oplus T_1^{(a)}=T_0^{(a)}\oplus U_{\mathbb T,N}$, and $T_0^{(a)}$ is a.c. unitary. Therefore, 
\begin{equation*} \mu_{T_0^{(a)}} + N =\mu_{T'^{(a)}}\leq \mu_{T'}=N. 
\end{equation*}
Consequently, $\mu_{T_0^{(a)}}=0$. This means that $T_0$ is a $C_0$-contraction. By {\cite[Theorem 0.1]{gamalaa03}}, 
$T'\prec T_0\oplus S_N$. By \cite{vasyuninkaraev},  
\begin{equation*} \mu_{T_0} + N =\mu_{T_0\oplus S_N}\leq \mu_{T'}=N. 
\end{equation*}
This means that $T_0$ acts on the zero space, that is, $T'=T_1\prec S_N$. 

Let $Y$ be a quasiaffinity such that $YS_N^*=T'^*Y$. Then $Y\ker S_N^*=\ker T^*$. 
Furthermore, 
\begin{align*} TY&=TT'^*YS_N=(I-P_{\ker T^*})YS_N=YS_N-P_{Y\ker S_N^*}YS_N\\&
=Y\Bigl(S_N +\sum_{k=1}^N e_k\otimes f_k\Bigr)
\end{align*}
for some $\{f_k\}_{k=1}^N\subset H^2_N$. By Lemma \ref{lemma1}, $S_N\buildrel i\over\prec (S_N +\sum_{k=1}^N e_k\otimes f_k)$. 
Thus, $S_N\buildrel i\over\prec T$.

Let $\mathcal M\in\operatorname{Lat}T$. If  $\dim( \mathcal M\ominus T\mathcal M)>N$, take a subspace 
$E\subset ( \mathcal M\ominus T\mathcal M)$ such that $\dim E=N+1$ and set  
$\mathcal N=\vee_{n=0}^\infty T^n E $.
Then $\dim\ker (T|_{\mathcal N})^*=N+1$. Applying to $T|_{\mathcal N}$ already proved part of the theorem, we obtain that $S_{N+1}\buildrel i\over\prec T|_{\mathcal N}$. 
Thus, $ S_{N+1}\buildrel i\over\prec T|_{\mathcal N}\buildrel i\over\prec T\prec S_N$, a contradiction.
\end{proof}

\section{Similarity to isometry} 

In this section, the relationship between similarity to isometry of an operator $T$ and its Cauchy dual $T'$ is studies. 

\begin{proposition}\label{prop1} Suppose that  $V$ and $V_1$ are isometries,  $T$ is a left-invertible operator,
 $T\approx V$ and $T'\approx V_1$. 
Then  $V\cong V_1$.
\end{proposition}

\begin{proof} Since $\dim\ker V^*=\dim\ker T^*=\dim\ker T'^*=\dim\ker V_1^*$, we conclude that there exist $0\leq N\leq\infty$ 
and unitaries $U\in\mathcal L(\mathcal K)$  and $U_1\in\mathcal L(\mathcal  K_1)$
 such that $V\cong U\oplus S_N$ and $V_1\cong U_1\oplus S_N$. 
Since $T\approx  U\oplus S_N$, there exists $\mathcal M\in \operatorname{Lat}T$ such that $T|_{\mathcal M}\approx U$.
Since $T'^*T=I$, we have $\mathcal M\in \operatorname{Lat}T'^*$ and  $T'^*|_{\mathcal M}\approx U^{-1}$.
Therefore, there exists $\mathcal N\in \operatorname{Lat}( U_1\oplus S_N)^*$ such that 
$( U_1\oplus S_N)^*|_{\mathcal N}\approx U^{-1}$. Since 
\begin{equation*} \ker P_{\mathcal K_1}|_{\mathcal N}=\mathcal N\cap H^2_N=\{0\}\ \ \text{ and } \ \ 
  P_{\mathcal K_1}|_{\mathcal N}\in\mathcal I(( U_1\oplus S_N)^*|_{\mathcal N}, U_1^{-1}), \end{equation*} 
we obtain 
that $U^{-1}\buildrel i\over\prec  U_1^{-1}$. 

Since $T''=T$, we can apply already proved result and obtain that $U_1^{-1}\buildrel i\over\prec  U^{-1}$.  
Consequently, $U\cong U_1$. 
\end{proof}

\begin{proposition} \label{prop2} Suppose that  $V$ is an isometry, $T$ is expansive, and  $T\approx V$. 
Then $T'\approx V$. 
\end{proposition}

\begin{proof} Since $T\approx V$, we have $C=\sup_{n\in\mathbb N}\|T^n\|<\infty$. Denote by $\mathcal H$
 the space on which $T$ acts.  For $x\in\mathcal H$ and $n\in\mathbb N$ we have
\begin{equation*} \|x\|=\|T^{*n}T'^nx\|\leq\|T^{*n}\|\|T'^nx\|\leq C \|T'^nx\|.
\end{equation*}
Since $T'$ is a contraction, the estimate $\inf_{n\in\mathbb N}\|T'^nx\|\geq\frac{1}{C}\|x\|$ ($x\in\mathcal H$) implies that 
$T'$ is similar to an isometry. By Proposition \ref{prop1}, $T'\approx V$.
\end{proof}

The following two examples show that an expansive operator $T$ in  Proposition \ref{prop2} cannot be replaced by contraction. 
In Example \ref{exa53} $T$ is an expansive operator such that $T'\approx S$, $T\sim S$ and $T\not\approx S$. 
In Example \ref{exa54} $T$ is an expansive operator such that $T'\approx S$ and $\mathcal R^\infty(T)\neq \{0\}$. 
The following result from \cite{nakamura93} will be used. 

 Let $g\in H^2$ be such that $\|g\|=1$ and $0<|g(0)|<1$. Set 
\begin{equation}\label{ttg} T=S-\mathbf{1}\otimes S^*\frac{g}{g(0)}. \end{equation} 
Then $T'=S-g\otimes S^*g$. Let $\omega$ be defined by \eqref{nakamura1} applied to $g$. By {\cite[Theorem 5]{nakamura93}}, 
\begin{equation*} \left[\begin{matrix} \omega\\(1-\omega)g\end{matrix}\right]
\end{equation*}
is the characteristic function of $T'$ (see {\cite[Ch. VI]{nfbk}} for the characteristic function of a contraction). 

\begin{example}\label{exa53} Let $g\in H^2$ be such that $\|g\|=1$, $|g(0)|<1$,  and $1/g\in H^\infty$. 
Define $\omega$ by \eqref{nakamura1} and $T$ by \eqref{ttg}.
Then 
\begin{equation*}\omega+\frac{1}{g}(1-\omega)g=1.\end{equation*}
 By \cite{nf73} or \cite{nf76}, 
$T'\approx S$. Furthermore, by   Lemma \ref{lemgsimss}(ii), $T\sim S$. Indeed, $T_{\frac{\overline{g(0)}g}{g(0)\overline g}}$ 
is a quasiaffinity, because $g$, $1/g\in H^2$, and $T_\varphi T_{\frac{\overline g}{\overline{g(0)}}}$
 is a quasiaffinity for some approriate $\varphi$, 
because $g$ is outer.  

 By Lemma \ref{lemgsimss}(iii), $T$ is similar to an isometry if and only if $T_{\frac{\overline{g(0)}g}{g(0)\overline g}}$ is invertible,
 what is of course equivalent that $T_{\frac{g}{\overline g}}$ is invertible. If  $T_{\frac{g}{\overline g}}$ is invertible, then by 
{\cite[Corollary 3.2.2]{peller}} there exists $p>2$ and $f\in H^p$ such that $1/f\in H^p$ and 
 $\frac{g}{\overline g}=\frac{\overline f}{f}$. Since $gf\in H^1$ and $gf=\overline{gf}$, 
we conclude that $gf\equiv c$ for some $c\in\mathbb C$. Thus, if $T_{\frac{g}{\overline g}}$ is invertible, then 
there exists $p>2$ such that $g$, $1/g\in H^p$.
Consequently, if $g\not\in H^p$ for any $p>2$, then 
 $T$ is not similar to an isometry.  The function  $g$ such that  $1/g\in H^\infty$ and 
 $g\not\in H^p$ for any $p>2$ is given in {\cite[Sec. 7, Example]{zheng}}. 
Namely, let $g_0$ be the outer function such that 
\begin{equation*}|g_0(\mathrm{e}^{\mathrm{i}\pi t})|=\frac{1}{|t|^{\frac{1}{2}}\log\frac{2}{|t|}},\ \ \ t\in(-1,0)\cup(0,1), 
\end{equation*}
and $g=g_0/\|g_0\|$.
\end{example}

\begin{example}\label{exa54} Let $f\in H^2$ be a nonconstant  function such that $\|f\|=1$, $|f|^2\in L^2$, 
and $1/f$, $P_+|f|^2\in H^\infty$. 
For example,  it is sufficient to take  $f$ which is analytic on $\mathbb D$, continuous on $\overline{\mathbb D}$, and 
such that $f(z)\neq 0$ and $|f(z)-f(w)|\leq C|z-w|$ for every $z$, $w\in\overline{\mathbb D}$ and some constant $C$. Let $\omega$ be defined by \eqref{nakamura1} applied to $f$. 
Then $\frac{1}{1-\omega}=P_+|f|^2$. Since $\omega\not\equiv 0$,
 there exist $\varphi_1$, $\varphi_2$, $\theta\in H^\infty$ such that 
$1\not\equiv\theta$ is inner and $\varphi_1\theta+\varphi_2\omega=1$ (see, for example,
the proof of {\cite[Prop. 5.3] {gamal22}}). Since $\omega(0)=0$, we have $\theta(0)\neq 0$.  Note that $\theta$ can be chosen such that $\dim\mathcal K_\theta=\infty$.
 Set $g=\theta f$ and define $T$ by \eqref{ttg}. 
Then
\begin{equation*}  \varphi_2\omega+\frac{1}{1-\omega}\frac{1}{f}\varphi_1(1-\omega)g=1. \end{equation*}
 By \cite{nf73} or \cite{nf76}, 
$T'\approx S$. By Lemma \ref{lemgsimss}(i), $\mathcal R^\infty(T)=\mathcal K_\theta\neq \{0\}$.
\end{example}
 
The following example shows that an expansive operator $T$ in  Proposition \ref{prop2} cannot be replaced by  an operator similar to expansive one.

\begin{example} Suppose that  $g\in H^2$,   $g(0)=1$, and $S^*g\not\equiv 0$. 
Set $E=\mathbf{1}\vee g$, $d_1=(g-1)/\|S^*g\|$ and  $d_2=\mathbf{1}$. 
Then $\{d_1,d_2\}$ is an orthonormal basis of $E$.
Take $a>\|S^*g\|^2$. Let
\begin{equation*} 
Y_0=\left[\begin{matrix}a & \|S^*g\| \\ \|S^*g\| & 1\end{matrix}\right]
\end{equation*}
be the matrix of the positive invertible operator  $Y_0\in\mathcal L(E)$ in the basis $\{d_1,d_2\}$. 
Set $Y=I_{H^2\ominus E}\oplus Y_0$. Then $Y\in\mathcal L(H^2)$ is a positive invertible operator, and $Y\mathbf{1}=g$.

Set $X=Y^{-\frac{1}{2}}$ and  $T=XSX^{-1}$. Then $T'=X^{-1}(S-\mathbf{1}\otimes S^*g)X$.  
Indeed, 
\begin{equation*} T'^*T=X(S^*-S^*g\otimes\mathbf{1})X^{-1}XSX^{-1}=X(S^*-S^*g\otimes\mathbf{1})SX^{-1}=XX^{-1}=I,
\end{equation*}
\begin{equation*} \ker T^*=\{h\in H^2: Xh=c \text{ for some } c\in\mathbb C\},
\end{equation*}
and
\begin{align*} \ker T'^* & =\{h\in H^2: X^{-1}h=cg \text{ for some }  c\in\mathbb C\}\\&=
\{h\in H^2: Y^{\frac{1}{2}}h=cY\mathbf{1}\text{ for some } c\in\mathbb C\} \\ & =
\{h\in H^2: h=cY^{\frac{1}{2}}\mathbf{1}\text{ for some }  c\in\mathbb C\}= \ker T^*.
\end{align*}
If  $T_{\frac{g}{\overline g}}$ is not  invertible, then, by Lemma \ref{lemgsimss}, $T'$ is not similar to an isometry.
\end{example}

\section{Quasisimilarity of expansive operator to the unilateral shift does not preserve the lattice of invariant subspaces}

 If $T$ is a contraction and $T\sim S$, then the intertwining quasiaffinities give a bijection between 
$\operatorname{Lat}T$ and $\operatorname{Lat}S$ (see \cite{gamal02} for more general result, see also references therein). 
 In particular, 
for every quasiaffinity $Y\in\mathcal I(S,T)$  
and every $\{0\}\neq\mathcal M\in \operatorname{Lat}T$ there exists an inner function $\vartheta$ such that 
 $\mathcal M=\operatorname{clos}Y\vartheta H^2$. In this section, it is shown that a contraction $T$ cannot be replaced by an expansive operator 
(Corollary \ref{cor6}). 

Recall that for an inner function $\vartheta$ the space $\mathcal K_\vartheta$ is defined in \eqref{30}.

\begin{lemma}\label{lemthetabetadense} Suppose that $\theta$ and $\beta$ are inner functions. Then 
\begin{equation}\label{kkthetabeta} \mathcal K_{\theta\beta}=\theta \mathcal K_\beta\oplus\mathcal K_\theta\end{equation} and 
\begin{equation}\label{kksubset} (\theta-1)\mathcal K_\beta\subset \mathcal K_{\theta\beta}.\end{equation}
Moreover, 
\begin{equation} \label{6eq} \operatorname{clos}(\theta-1)\mathcal K_\beta= \mathcal K_{\theta\beta}\end{equation}
and if and only if 
\begin{equation} \label{6cap}(\mathcal K_\theta+(\theta-1)\mathcal K_\beta)\cap \beta H^2=\{0\}.\end{equation}
\end{lemma}

\begin{proof} The equality \eqref{kkthetabeta} follows from the definition of the space $\mathcal K_\vartheta$ for an inner function $\vartheta$ (see \eqref{30}), 
and the inclusion \eqref{kksubset}  easy follows from \eqref{kkthetabeta}. Let $f\in \mathcal K_\theta$ and $h\in\mathcal K_\beta$ be such that 
$0\not\equiv\theta h+f\perp (\theta-1)\mathcal K_\beta$. Then there exist $h_1\in H^2$ and $h_2\in H^2_-$ such that 
\begin{equation*} (1-\theta)h+(\overline\theta-1)f=\beta h_1+h_2.\end{equation*}
By \eqref{30}, $\overline\theta f\in H^2_-$. Consequently, $(1-\theta)h-f=\beta h_1$. By \eqref{31},
 the relation $(1-\theta)h-f\equiv 0$ implies
$h\equiv 0$ and $f\equiv 0$, a contradiction with the assumption on $h$ and $f$. Thus, if \eqref{6eq} is not fulfilled, then \eqref{6cap} 
is not fulfilled.

Conversely, let $f\in \mathcal K_\theta$,  $h\in\mathcal K_\beta$, and  $h_1\in H^2$ be such that $0\not\equiv(1-\theta)h-f=\beta h_1$. 
Then \begin{equation*} (\overline\theta-1)\theta h + (\overline\theta-1)f=\beta h_1+\overline\theta f.\end{equation*}
By \eqref{30}, $(\overline\theta-1)\theta h + (\overline\theta-1)f\perp\mathcal K_\beta$. Consequently, 
$\theta h+f\perp (\theta-1)\mathcal K_\beta$. If $\theta h+f\equiv 0$, then $h\equiv 0$ and $f\equiv 0$, 
 a contradiction with the assumption on $h$ and $f$. Thus, if \eqref{6cap} is not fulfilled, then \eqref{6eq} 
is not fulfilled.
\end{proof}

For $\zeta\in\mathbb T$ and $t_0>0$ set 
\begin{equation}\label{ddelta}\Delta(\zeta, t_0)=\{\zeta\mathrm{e}^{\mathrm{i}t}\ :\ |t|\leq t_0\}.\end{equation} 
 For $\zeta\in\mathbb T$ and $0<s_0<\pi/2$ denote by $\mathcal S(\zeta,s_0)$ the Stolz angle, that is, the closed 
sector with vertex $\zeta$ of angle $2s_0$ 
symmetric with respect to the radius $\{r\zeta\ : r\in[0,1]\}$. (Usually, the Stolz angle assumed to be an open set, but it is convenient to consider closed set here.)  For  $0<r_0<1$ sufficiently close to $1$ 
both rays which form the boundary of $\mathcal S(\zeta,s_0)$ 
intersect the circle $\{|z|=r_0\}$ in two points.  Denote by $z_\pm$  two points from this intersection 
closest to $\zeta$. Define $t(s_0,r_0)$ as follows:  $z_\pm=r_0\zeta\mathrm{e}^{\pm\mathrm{i}t(s_0,r_0)}$. 
Then 
\begin{equation}\label{tsr}\tan s_0=\frac{r_0\sin t(s_0,r_0)}{1-r_0\cos t(s_0,r_0)}.\end{equation}
Set \begin{equation}\label{sszsr}
\mathcal S(\zeta,s_0,r_0)=\mathcal S(\zeta,s_0)\cap\{r\zeta\mathrm{e}^{\mathrm{i}t}\ :\ |t|\leq t(s_0,r_0),\  r_0\leq r\leq 1\}.\end{equation}
Then $\mathcal S(\zeta,s_0,r_0) $ is  a ``triangle" with vertices $\zeta$ and $r_0\zeta\mathrm{e}^{\pm\mathrm{i}t(s_0,r_0)}$; its two edges are segments and one edge is a subarc of 
the circle $\{|z|=r_0\}$.

The following simple lemma is given by convenience of references; its proofs is omitted.

\begin{lemma} \label{lemepsilon} Let $0<\varepsilon<1$, and let $\zeta_0\in\mathbb T$. Set $\lambda_0=(1-\varepsilon)\zeta_0$ and define
 $0<s(\varepsilon)<\pi/2$ by the formula 
\begin{equation*} \tan s(\varepsilon)=\frac{(1-\varepsilon)\sin\varepsilon}{1-(1-\varepsilon)\cos\varepsilon}.
\end{equation*}
Then $s(\varepsilon)\to\pi/4$ as $\varepsilon\to 0$. Furthermore, if  $\zeta\in\Delta(\zeta_0,\varepsilon)$, then 
$|\zeta-\lambda_0|^2\leq \varepsilon^2+2(1-\varepsilon)(1-\cos\varepsilon)$ and 
$\lambda_0\in\mathcal S(\zeta, s(\varepsilon))$.  
\end{lemma}

For $\lambda\in\mathbb D$, $\lambda\neq 0$,  a Blaschke factor is $b_\lambda(z)=\frac{|\lambda|}{\lambda}\frac{\lambda-z}{1-\overline\lambda z}$ 
($z\in\mathbb D$). If $\Lambda\subset\mathbb D$ satisfies
 the Blaschke condition 
\begin{equation} \label{6blaschke}\sum_{\lambda\in\Lambda} (1-|\lambda|)<\infty, \end{equation}
then the Blaschke product  $\beta=\prod_{\lambda\in\Lambda} b_\lambda$ converges and $\beta$ is an inner function.

For $\zeta\in\mathbb T$ and  $0<s<1$ set
\begin{equation*} Q(\zeta, s)=\{z\in\mathbb D\ :\ 1-s\leq |z|<1, \ \frac{z}{|z|}\in\Delta(\zeta, s)\}.
\end{equation*} 
The set $Q(\zeta, s)$ is called the Carleson box or the Carleson window. Let $\Lambda\subset\mathbb D$ satisfy \eqref{6blaschke}.
A particular case of the Carleson embedding theorem 
(see, for example, {\cite[Theorem 11.22]{gmr}}) is the following: the relations 
\begin{equation} \label{6carleson} \sum_{\lambda\in\Lambda} |h(\lambda)|^2(1-|\lambda|)<\infty \ \text{ for every }h\in H^2
\end{equation}
and 
\begin{equation} \label{6carlesonqq} 
\sup_{\zeta\in\mathbb T\atop 0<s<1}\frac{1}{s}\sum_{\lambda\in\Lambda\cap Q(\zeta, s)} (1-|\lambda|)<\infty
\end{equation}
are equivalent.

\begin{lemma}\label{lem61} Suppose that $0<s_0<\pi/2$,  $\Lambda\subset\mathbb D$, $\Lambda$ satisfies \eqref{6carleson}, 
$\nu$ is a singular positive Borel measure on $\mathbb T$,
and 
\begin{equation*}
\nu(\{\zeta\in\mathbb T\ :\ \zeta\in\operatorname{clos}(\Lambda\cap\mathcal S(\zeta,s_0))\}= \nu(\mathbb T)=1.
\end{equation*}
Set $\beta=\prod_{\lambda\in\Lambda}b_\lambda$, and define $\theta$ by \eqref{thetaclark}. Then $\theta$ and $\beta$ satisfy \eqref{6cap}.
\end{lemma}

\begin{proof} Let $C>0$,  and let $\{t_j\}_{j=1}^\infty$ be such that $t_j>0$ and $t_j\to 0$. Set 
\begin{equation*}\sigma_j=\{\zeta\in\mathbb T\ : \ \nu(\Delta(\zeta, t))\geq Ct \ \text{ for all } 0<t\leq t_j\}\ \ (j\geq 1).
\end{equation*}
Since $\nu$ is a singular measure, we have $\nu(\mathbb T)=\nu(\cup_{j=1}^\infty\sigma_j)$ 
(see, for example, {\cite[Theorem 1.2]{gmr}} or  {\cite[formula (II.6.3)]{garnett}}).

Let $0\not\equiv f\in\mathcal K_\theta$. 
By \eqref{yy1}, 
$f$ has nontangential boundary values $f(\zeta)$ for $\nu$-a.e. $\zeta\in\mathbb T$, and 
$\nu(\{\zeta\in\mathbb T\ :\ f(\zeta)\neq 0\})>0$. Therefore, there exist $\delta>0$ and $0<r_0<1$ such that 
$\nu(\tau_0)>0$, where
\begin{equation*}\tau_0=\{\zeta\in\mathbb T\ : \ |f(z)|\geq\delta \ \text{ for every } z\in \mathcal S(\zeta,s_0,r_0)\}
\end{equation*}
and  $\mathcal S(\zeta,s_0,r_0)$ is defined in \eqref{sszsr}.
Indeed, let $\{\delta_n\}_{n=1}^\infty$ and $\{r_n\}_{n=1}^\infty$ be such that $\delta_n>0$, $0<r_n<1$,  
$\delta_n\to 0$, 
and $r_n\to 1$. Set 
\begin{equation*}\tau_{nk}=\{\zeta\in\mathbb T\ : \ |f(z)|\geq\delta_n \ \text{ for every } z\in \mathcal S(\zeta,s_0,r_k)\}.
\end{equation*}
Then $\{\zeta\in\mathbb T\ :\ f(\zeta)\neq 0\}=\cup_{n,k=1}^\infty \tau_{nk}$. Consequently, there exist $n$ and $k$ such that 
$\nu(\tau_{nk})>0$. Set $\delta=\delta_n$, $r_0=r_k$, and $\tau_0=\tau_{nk}$. Furthermore, there exists $j$ such that 
$\nu(\tau_0\cap\sigma_j)>0$. Set $t_0=t_j$ and 
\begin{equation*}\tau=\tau_0\cap\sigma_j\cap
\{\zeta\in\mathbb T\ :\ \zeta\in\operatorname{clos}(\Lambda\cap\mathcal S(\zeta,s_0))\}.\end{equation*} 
Then $\nu(\tau)>0$.

Let $h\in H^2$ be such that $f+(\theta-1)h\in\beta H^2$. Then $h(\lambda)=f(\lambda)/(1-\theta(\lambda))$ for every $\lambda\in\Lambda$.
The equality \eqref{thetaclark} implies that 
\begin{equation*} \frac{1}{|\theta(z)-1|}\geq\int_{\mathbb T}\mathrm{Re}\frac{1}{1-z\overline\zeta}\mathrm{d}\nu(\zeta)
\geq\frac{1-|z|\cos t}{1-2|z|\cos t+|z|^2}\nu\Bigl(\Delta\Bigl(\frac{z}{|z|}, t\Bigr)\Bigr)\end{equation*}
for every $z\in\mathbb D$ and $0<t<\pi/2$.

Let $\lambda\in\Lambda\cap\mathcal S(\zeta,s_0,r_0)$ for some $\zeta\in\tau$. Then 
\begin{equation*} |h(\lambda)|=\frac{|f(\lambda)|}{|\theta(\lambda)-1|}\geq 
\delta\frac{1-|\lambda|\cos t}{1-2|\lambda|\cos t+|\lambda|^2}\nu\Bigl(\Delta\Bigl(\frac{\lambda}{|\lambda|}, t\Bigr)\Bigr)  \end{equation*}
for every  $0<t<\pi/2$. Set $t(\lambda)=2t(s_0,|\lambda|)$, where $t(s_0,r)$ is defined before \eqref{tsr} for $0<s_0<\pi/2$ and $0<r<1$ sufficiently close to $1$.  Then
\begin{equation}\label{6c} \begin{aligned}t(\lambda) & \sim 2(1-|\lambda|)\tan s_0 \ \text{ and }\\
 \frac{1-|\lambda|\cos t(\lambda) }{1-2|\lambda|\cos t(\lambda) +|\lambda|^2} & \sim \frac{1}{1+4\tan^2s_0}\frac{1}{1-|\lambda|}
 \ \text{ as } |\lambda|\to 1\end{aligned}
\end{equation}
(where $a(t)\sim b(t)$ as $t\to c$ means that $\lim_{t\to c}a(t)/b(t)=1$).
Furthermore, 
\begin{equation*} \Delta(\zeta, t(s_0,|\lambda|))\subset\Delta\Bigl(\frac{\lambda}{|\lambda|}, t(\lambda)\Bigr).
\end{equation*}
It follows from this inclusion and the  construction of $\tau$ that 
 $\nu(\Delta(\frac{\lambda}{|\lambda|}, t(\lambda)))\geq Ct(s_0,|\lambda|)=Ct(\lambda)/2$, if $t(s_0,|\lambda|)\leq t_0$. 
By \eqref{6c}, there exists $0<c<1$ (which does not depend on $\lambda$) such that 
\begin{align*} 
&|h(\lambda)|^2(1-|\lambda|)\\& 
\geq c \delta^2 C (1-|\lambda|)\tan s_0 \frac{1}{(1+4\tan^2s_0)^2}\frac{1}{(1-|\lambda|)^2}(1-|\lambda|)
\nu\Bigl(\Delta\Bigl(\frac{\lambda}{|\lambda|}, t(\lambda)\Bigr)\Bigr)\\&= c \delta^2 C \tan s_0 \frac{1}{(1+4\tan^2s_0)^2}
\nu\Bigl(\Delta\Bigl(\frac{\lambda}{|\lambda|},t(\lambda)\Bigr)\Bigr).
\end{align*}
Set $C_1=c\delta^2 C \tan s_0 \frac{1}{(1+4\tan^2s_0)^2}$. Let $0<\varepsilon<1-r_0$.
Then 
\begin{align*} 
\sum_{\lambda\in\Lambda: |\lambda|\geq 1-\varepsilon}&|h(\lambda)|^2(1-|\lambda|)\geq 
\sum_{\lambda\in\Lambda\cap(\cup_{\zeta\in\tau}\mathcal S(\zeta,s_0,r_0))
: \atop |\lambda|\geq 1-\varepsilon}|h(\lambda)|^2(1-|\lambda|) \\&
\geq \sum_{\lambda\in\Lambda\cap(\cup_{\zeta\in\tau}\mathcal S(\zeta,s_0,r_0))
:\atop |\lambda|\geq 1-\varepsilon} C_1\nu\Bigl(\Delta\Bigl(\frac{\lambda}{|\lambda|},t(\lambda)\Bigr)\Bigr)\\&
\geq C_1 \nu\Bigl(\bigcup_{\lambda\in\Lambda\cap(\cup_{\zeta\in\tau}\mathcal S(\zeta,s_0,r_0))
:\atop |\lambda|\geq 1-\varepsilon}\Delta\Bigl(\frac{\lambda}{|\lambda|},t(\lambda)\Bigr)\Bigr)
\geq C_1 \nu(\tau),
\end{align*} 
because for every $\zeta\in\tau$  there exists $\lambda\in \Lambda\cap\mathcal S(\zeta,s_0,r_0)$ with $|\lambda|\geq 1-\varepsilon$. 
Consequently, 
\begin{equation*} \lim_{\varepsilon\to 0}\sum_{\lambda\in\Lambda: |\lambda|\geq 1-\varepsilon}
|h(\lambda)|^2(1-|\lambda|)\geq C_1 \nu(\tau)>0,
\end{equation*}
a contradiction with \eqref{6carleson}. Thus, if $f\in\mathcal K_\theta$ and $h\in H^2$ are such that $f+(\theta-1)h\in\beta H^2$, then  
$f\equiv 0$ and $(\theta-1)h\in\beta H^2$. Since $\theta-1$ is outer, we have $h\in\beta H^2$. If $h\in\mathcal K_\beta$, then $h\equiv 0$.
\end{proof}

The following simple lemma is given by convenience of references; its proofs is omitted.

\begin{lemma}\label{lemkk} Let $K\subset \mathbb T$ be a  compact, and let $\delta>0$. 
 If $m(K)=0$, then there exist $N\in\mathbb N$
 and nonempty closed subarcs 
 $\{\Delta_k\}_{k=1}^N$ such that $\Delta_k\subset\mathbb T$, $\Delta_k\cap\Delta_j=\emptyset$, if $k\neq j$ $(1\leq k,j\leq N)$, $K\subset\cup_{k=1}^N\Delta_k$  and $\sum_{k=1}^N\pi m(\Delta_k)<\delta$. 
\end{lemma}

\begin{lemma} \label{lemcarlesonkk}
 Suppose that $\{K_n\}_{n=1}^\infty$ is a sequence of compact subsets of $\mathbb T$ 
such that $K_n\subset K_{n+1}$ and 
$m(K_n)=0$ for all $n=1,2, \ldots$. Then there exists $\Lambda\subset\mathbb D$  which satisfies \eqref{6blaschke} and  \eqref{6carlesonqq}  and
 such that
\begin{equation} \label{nontang} \zeta\in\operatorname{clos}(\Lambda\cap\mathcal S(\zeta,s))\end{equation} 
 for every $\zeta\in\cup_{n=1}^\infty K_n$ and $\pi/4<s<\pi/2$.
\end{lemma}

\begin{proof} 
Take $\{\delta_k \}_{k=1}^\infty$ such that $0<\delta_k<1$ for all $k=1,2, \ldots$, 
and $\sum_{k=1}^\infty \delta_k<\infty$. 
Let $\{\Delta_{1k}\}_{k=1}^{N_1}$ are subarcs from Lemma \ref{lemkk} applied to $K_1$ and $\delta_1$. 
Set $\varepsilon_{1k}=\pi m(\Delta_{1k})$ ($k=1,\ldots, N_1$). We may assume that 
\begin{equation*} \varepsilon_{1N_1}=\min_{k=1,\ldots, N_1}\varepsilon_{1k}. \end{equation*}
There exists $M_1>1$ such that $\sum_{k=M_1}^\infty \delta_k< \varepsilon_{1N_1}$.
Let $\{\Delta_{2k}\}_{k=1}^{N_2}$ are subarcs from Lemma \ref{lemkk} applied to $K_2$ and $\delta_{M_1}$. 
Set   $\varepsilon_{2k}=\pi m(\Delta_{2k})$ ($k=1,\ldots, N_2$). We may assume that 
\begin{equation*} \varepsilon_{2N_2}=\min_{k=1,\ldots, N_2}\varepsilon_{2k}. \end{equation*}
There exists $M_2>M_1$ such that $\sum_{k=M_2}^\infty \delta_k< \varepsilon_{2N_2}$. 
Let $\{\Delta_{3k}\}_{k=1}^{N_3}$ are subarcs from Lemma \ref{lemkk} applied to $K_3$ and $\delta_{M_2}$, and so on. 
Set $M_0=1$. 
We obtain  the sequence $\{M_n\}_{n=0}^\infty\subset\mathbb N$, the 
 closed subarcs $\Delta_{nk}$ and the quantities $\varepsilon_{nk}>0$ ($n=1,2, \ldots$, $k=1,\ldots, N_n$). 
Note that \begin{gather*}
\varepsilon_{nN_n}\leq\varepsilon_{nk}<\varepsilon_{n-1,N_{n-1}} \ \ (k=1, \ldots, N_n) \ \ \text{and }\\  \sum_{k=1}^{N_n}\varepsilon_{nk}<\delta_{M_{n-1}}\ \ (n=1,2,\ldots).  \end{gather*}

Define $\lambda_{nk}\in\mathbb D$ such that 
\begin{equation*} \Delta_{nk}=\Delta\Bigl(\frac{\lambda_{nk}}{|\lambda_{nk}|}, 1-|\lambda_{nk}|\Bigr) \ \ (n=1,2, \ldots, \ k=1,\ldots, N_n),
 \end{equation*}
where $\Delta(\zeta, s)$ is defined by \eqref{ddelta}. Then $1-|\lambda_{nk}|=\varepsilon_{nk}$ ($n=1,2, \ldots$, $k=1,\ldots, N_n$). 
Set $\Lambda=\{\lambda_{nk}, \ n=1,2, \ldots,\ k=1,\ldots, N_n\}$. 
We will to prove that $\Lambda$ satisfies the conclusion of the lemma.

The relation \eqref{6blaschke} easy follows from the construction. 

Let $\zeta\in \cup_{n=1}^\infty K_n$. Then there exists $q\in\mathbb N$ such that $\zeta\in  K_n$ for all $n\geq q$. By construction, 
for every $n\geq q$ there exists $1\leq k\leq N_n$ such that $\zeta\in\Delta_{nk}$. By Lemma \ref{lemepsilon}, 
$|\zeta-\lambda_{nk}|\to 0$ when $n\to\infty$. Define $s(\varepsilon_{nk})$ as in Lemma \ref{lemepsilon}. Then 
  $\lambda_{nk}\in\mathcal S(\zeta, s(\varepsilon_{nk}))$ and $s(\varepsilon_{nk})\to\pi/4$ when $n\to\infty$. Consequently,  the relation \eqref{nontang} is fulfilled. 

Let $\zeta\in\mathbb T$, and let  $0<s<1$ be sufficiently close to $0$. Then there exists  $q\in\mathbb N$ such that 
$\varepsilon_{qN_q}\leq s <\varepsilon_{q-1, N_{q-1}}$. 
 Let $\lambda_{nk}\in Q(\zeta, s)$. Then $|\lambda_{nk}|\geq 1-s$. 
Consequently, $n\geq q$. We have \begin{equation}
\begin{aligned}\label{6lem1}\frac{1}{s}\sum_{n= q+1}^\infty \sum_{k=1}^{N_n} (1-|\lambda_{nk}|)&=
\frac{1}{s}\sum_{n= q+1}^\infty \sum_{k=1}^{N_n} \varepsilon_{nk}\\&
\leq\frac{1}{s}\sum_{n= q+1}^\infty \delta_{M_{n-1}}
\leq\frac{1}{s}\sum_{k=M_q}^\infty \delta_k\leq\frac{1}{s} \varepsilon_{qN_q}\leq 1.
\end{aligned}\end{equation}
Clearly, if $\lambda_{qk}\in Q(\zeta, s)$, then $\frac{\lambda_{qk}}{|\lambda_{qk}|}\in \Delta(\zeta, s)$. 
Since $\Delta_{qk}\cap\Delta_{qj}=\emptyset$, if $k\neq j$ ($1\leq k,j\leq N_q$), we have 
\begin{equation*} \operatorname{card}\{k\ : \ 1\leq k\leq N_q, \lambda_{qk}\in Q(\zeta, s), \ \Delta_{qk}\not\subset \Delta(\zeta, s)\}\leq 2.
\end{equation*}
Therefore, \begin{equation}
\begin{aligned}\label{6lem2}\frac{1}{s} &\sum_{1\leq k\leq N_q: \atop \lambda_{qk}\in Q(\zeta, s)} (1-|\lambda_{qk}|) 
\leq
\frac{1}{s}\Bigl(2s + \sum_{1\leq k\leq N_q: \atop \Delta_{qk}\subset  \Delta(\zeta, s)} (1-|\lambda_{qk}|)\Bigr)\\&
=
\frac{1}{s}\Bigl(2s + \sum_{1\leq k\leq N_q: \atop \Delta_{qk}\subset  \Delta(\zeta, s)} \pi m(\Delta_{qk})\Bigr)\leq 
\frac{1}{s}\Bigl(2s +  \pi m(\Delta(\zeta, s))\Bigr)=3.
\end{aligned}\end{equation}
The relation  \eqref{6carlesonqq} follows from \eqref{6lem1} and \eqref{6lem2}.
\end{proof}

\begin{remark} In Lemma \ref{lemcarlesonkk} it is possible that $K_n=K_{n+1}$ for all $n=1,2, \ldots$.\end{remark}

\begin{theorem}\label{thm6main} Let $\theta$ be an inner function such that $\theta(0)=0$. Then there exists 
a Blaschke product $\beta$ with simple zeroes such that $\theta$ and $\beta$ satisfy \eqref{6eq}, and the set $\Lambda$ of zeroes of $\beta$   
satisfies \eqref{6carleson}.
\end{theorem}

\begin{proof} Let $\nu$ be defined by \eqref{thetaclark}. Since $\nu$ is singular with respect to $m$, there exists a sequence $\{K_n\}_{n=1}^\infty$ 
of compact subsets of $\mathbb T$ 
such that $K_n\subset K_{n+1}$, 
$m(K_n)=0$ for all $n=1,2, \ldots$, and $\nu(\cup_{n=1}^\infty K_n)=\nu(\mathbb T)$. Let $\Lambda$ 
 be the set from Lemma \ref{lemcarlesonkk} applied to $\{K_n\}_{n=1}^\infty$. Then  $\Lambda$    
satisfies \eqref{6carleson}, because \eqref{6carleson} and \eqref{6carlesonqq} are equivalent
 when \eqref{6blaschke} is fulfilled (see the reference before \eqref{6carleson}).  Set $\beta=\prod_{\lambda\in\Lambda}b_\lambda$. By Lemma \ref{lem61}, $\theta$ and $\beta$ satisfy \eqref{6cap}. 
By Lemma \ref{lemthetabetadense}, $\theta$ and $\beta$ satisfy \eqref{6eq}. 
\end{proof}

The following simple lemma is given for convenience of references; its proof is omitted. 

\begin{lemma} \label{lem6kerdense} Let $X\in\mathcal L(\mathcal H,\mathcal K)$ have the form 
 \begin{equation*} X=\left[\begin{matrix} X_1 & * \\ \mathbb O  & X_0\end{matrix}\right]
\end{equation*} with respect to some decompositions $\mathcal H=\mathcal H_1\oplus \mathcal H_0$ and 
$\mathcal K=\mathcal K_1\oplus \mathcal K_0$. Then  \begin{enumerate}[\upshape (i)]

\item if $\ker X_0=\{0\}$ and $\ker X_1=\{0\}$, then $\ker X=\{0\}$;

\item if  $\operatorname{clos}X_0\mathcal H_0=\mathcal K_0$ and  $\operatorname{clos}X_1\mathcal H_1=\mathcal K_1$, then 
$\operatorname{clos}X\mathcal H=\mathcal K$.
\end{enumerate}
\end{lemma}

The following theorem is the main result of this section. 

\begin{theorem} \label{thm6inv} Suppose that  $\theta$, $\beta\in H^\infty$ are inner functions,  $\theta(0)=0$, and 
$a$, $b\in\mathbb C$.  
Set $\varphi=a+b(\theta-1)$. Suppose that $\varphi\not\equiv 0$. Define $T$, $X$, $Y\in\mathcal L(H^2)$ as follows:
\begin{equation*} T=S+(1+(a-1)\theta)\beta\otimes\beta\overline\chi\theta+b\theta\beta\otimes P_+\overline\chi\beta,
\end{equation*}
\begin{equation*} X(\theta\beta h+\beta f+g)=(\theta-1)\beta h+a\beta f + \varphi g 
\ \ \ (h\in H^2, \ f\in\mathcal K_\theta, \ g\in\mathcal K_\beta),
\end{equation*}
\begin{equation*} Y(\beta h + g)=\theta\beta\varphi h +\theta P_{\beta H^2}\varphi g + (\theta-1)g 
\ \ \ (h\in H^2,  \ g\in\mathcal K_\beta).
\end{equation*}
Then $\theta\beta H^2$,  $\beta H^2\in\operatorname{Lat}T$, 
\begin{equation}\label{6uucong}P_{\beta\mathcal K_\theta}T|_{\beta\mathcal K_\theta}\cong U(\theta),\end{equation} 
where $U(\theta)$ is defined in \eqref{uutheta}, $YS=TY$, $XT=SX$, and $\ker Y=\{0\}$. 
Furthermore, \begin{enumerate}[\upshape (i)]

\item if $\varphi$ is outer and $a\neq 0$, then $X$ is a quasiaffinity;

\item if $\varphi$ is outer and \eqref{6eq} is fulfilled for $\theta$ and $\beta$, then $Y$ is a quasiaffinity;

\item $T$ is expansive if and only if $2\operatorname{Re}\overline ab\leq -1$, and then $1/\varphi\in H^\infty$. 
\end{enumerate}
\end{theorem}

\begin{proof} Many statements of the theorem can be checked directly.  Unitary equivalence in \eqref{6uucong} is given by 
the operator of  multiplication by $\beta$. 

The definition of $Y$ implies that  $Y$ has the form 
\begin{equation*} Y=\left[\begin{matrix} Y_1 & * \\ \mathbb O  & Y_0\end{matrix}\right]
\end{equation*} with respect to the decompositions $H^2=\beta H^2\oplus \mathcal K_\beta$ and 
 $H^2=\theta\beta H^2\oplus \mathcal K_{\theta\beta}$, where $Y_1$ and $Y_0$ are the operators of multiplication by 
$\theta\varphi$ and $\theta-1$, respectively. The equality $\ker Y=\{0\}$ and the statement (ii)  follow from Lemma \ref{lem6kerdense}. 

The definition of $X$ implies that  $X$ has the form 
\begin{equation*} X=\left[\begin{matrix} X_1 & * \\ \mathbb O  & X_0\end{matrix}\right]
\end{equation*} with respect to the decomposition
$H^2=\beta H^2\oplus \mathcal K_\beta$. The operator  $X_0\in\mathcal L(\mathcal K_\beta)$ acts by the formula 
 $X_0g=P_{\mathcal K_\beta}\varphi g$ ($g\in\mathcal K_\beta$), that is, $X_0=\varphi(S(\beta))$ 
(see {\cite[Theorem III.2.1]{nfbk}}). 
If $\varphi$ is outer, then $\varphi(S(\beta))$ is a quasiaffinity by {\cite[Prop. III.3.1]{nfbk}}. 
Furthermore, $X_1(\theta\beta h+\beta f)=(\theta-1)\beta h+a\beta f$ ($h\in H^2$, $ f\in\mathcal K_\theta$). Consequently, 
$\operatorname{clos}X_1\theta\beta H^2=\beta H^2$, because $\theta-1$ is outer. 
If $a\neq 0$, then $\ker X_1=\{0\}$ by \eqref{31}. The statement (i) follows from Lemma \ref{lem6kerdense}. 

A computation shows that $\{\beta\overline \chi\theta, \ P_+\overline \chi\beta/\|P_+\overline \chi\beta\|\}$ is an orthonormal basis of the  
range of $T^*T-I$, and 
\begin{equation*} A:=\left[\begin{matrix}|a|^2 & (\overline ab+1)\|P_+\overline \chi\beta\|\\
(a\overline b+1)\|P_+\overline \chi\beta\| & |b|^2\|P_+\overline \chi\beta\|^2\end{matrix}\right]
\end{equation*}
 is the matrix of the restriction of $T^*T-I$ on its range in this basis.   Thus, $T$ is expansive if and only if $A\geq 0$. 
Furthermore,  $A\geq 0$ if and only if $|a|^2|b|^2\geq |\overline ab+1|^2$,  
which is equivalent to the inequality $2\operatorname{Re}\overline ab\leq -1$. 
On the other hand, if $|a-b|>|b|$, then  $1/\varphi\in H^\infty$. Clearly, $|a-b|>|b|$ if and only if $|a-b|^2>|b|^2$, 
which is equivalent to the inequality $2\operatorname{Re}\overline ab< |a|^2$, which follows 
from the inequality $2\operatorname{Re}\overline ab\leq -1$.  
The statement (iii) is proved.
\end{proof}

\begin{corollary} \label{cor6} There exists an expansive operator $T$ with the following properties:  $T\sim S$, and there exists 
$\mathcal M\in \operatorname{Lat}T$ such that $\mathcal M\neq\operatorname{clos}Y\vartheta H^2$ for every $Y\in\mathcal I(S,T)$ and 
inner function $\vartheta$. 
\end{corollary}

\begin{proof} Take $a$, $b\in\mathbb C$ such that $2\operatorname{Re}\overline ab\leq -1$, and inner functions
$\theta$, $\beta\in H^\infty$ such that $\theta(0)=0$ and \eqref{6eq} is fulfilled for $\theta$ and $\beta$. 
(Such inner functions exist by Theorem \ref{thm6main}.) Define $T$ as in Theorem \ref{thm6inv} and set $\mathcal M= \beta H^2$. 
By Theorem \ref{thm6inv}, $T\sim S$. If $\mathcal M=\operatorname{clos}Y\vartheta H^2$ for some $Y\in\mathcal I(S,T)$, then 
$S\buildrel d\over\prec T|_{\mathcal M}$. Consequently, $(T|_{\mathcal M})^*\buildrel i\over\prec S^*$. 
By \eqref{6uucong}, $U(\theta)^* \buildrel i\over\prec S^*$, a contradiction (see \eqref{thetanucong}).
\end{proof}

\end{document}